\documentclass[a4paper,10pt]{amsart}

\usepackage{amstext}
\usepackage{amsmath}
\usepackage{ifthen}
\usepackage{amscd}
\usepackage{amssymb}
\usepackage[english]{babel}
\usepackage[T1]{fontenc}
\usepackage[utf8]{inputenc}
\usepackage[all]{xy}
\usepackage{graphicx}
\usepackage{enumerate}
\usepackage{xspace}
\usepackage{pdfsync}
\usepackage{epic}

\newtheorem{prop}{Proposition}[subsection]
\newtheorem{coro}[prop]{Corollary}
\newtheorem{lem}[prop]{Lemma}
\newtheorem{rem}[prop]{Remark}
\newtheorem{exe}[prop]{Example}
\newtheorem{defi}[prop]{Definition}
\newtheorem{theo}[prop]{Theorem}
\newtheorem{propdefi}[prop]{Proposition-Definition}
\newtheorem{conj}[prop]{Conjecture}

\newtheorem{ques}[prop]{Question}
\newtheorem{theoi}{Theorem}

\newtheorem{coroi}[theoi]{Corollary}
\newtheorem{conji}[theoi]{Conjecture}

\newcommand{\R}{\mathbb R}
\newcommand{\Q}{\mathbb Q}
\newcommand{\C}{\mathbb C}
\newcommand{\N}{\mathbb N}
\newcommand{\V}{\mathbb V}

\newcommand{\Z}{\mathbb Z}

\newcommand{\F}{\mathcal F}
\renewcommand{\L}{\mathcal L}
\newcommand{\M}{\mathcal M} 
\newcommand{\diff}{\mathrm{Diff_X}}
\newcommand{\Ox}{{\mathcal O}_{X} }
\newcommand{\Dx}{{\mathcal D}_{X} }
\newcommand{\omod}{\mathbf{Mod}({\mathcal O}_{X})}
\newcommand{\coh}{\mathbf{Coh}({\mathcal O}_{X})}
\newcommand{\dmod}{\mathbf{Mod}({\mathcal D}_{X})}
\newcommand{\filtdmod}{Filt\mathbf{Mod}({\mathcal D}_{X})}

\newcommand{\cohd}{\mathbf{Coh}({\mathcal D}_{X})}
\newcommand{\hol}{\mathbf{Hol}({\mathcal D}_{X})}

\newcommand{\wt}{\widetilde}

\newcommand{\Ng}{\mathcal{N}(\Gamma)}
\newcommand{\Ug}{\mathcal{U}(\Gamma)}

\newenvironment{prf}{Proof:}{$\Box$}

\begin{document}

\title[$L^2$-Invariants of coherent $\mathcal{D}$-modules and Hodge Modules]{Towards a $L_2$ cohomology theory for Hodge Modules on infinite covering spaces: $L_2$ Constructible Cohomology and $L_2$ De Rham Cohomology for coherent $\mathcal{D}$-modules.  }

\date{\today}

\author[P.Eyssidieux]{Philippe Eyssidieux}
\address{Institut Fourier,
Laboratoire de Mathematiques UMR 5582,   
Universit\'e Grenoble Alpes,
CS 40700, 
38058 Grenoble, France}
\email{philippe.eyssidieux@univ-grenoble-alpes.fr}
\urladdr{http://www-fourier.ujf-grenoble.fr/$\sim$eyssi/}
\thanks{This  research was partially supported by the ANR project Hodgefun ANR-16-CE40-0011-01.}

\begin{abstract} 

This article constructs Von Neumann invariants for constructible complexes and coherent $\mathcal{D}$-modules on compact complex manifolds,  generalizing the  work of the author on coherent $L_2$-cohomology. 
We formulate a  conjectural generalization of Dingoyan's $L_2$-Mixed Hodge structures   in terms of Saito's Mixed Hodge Modules and give partial results in this direction. 

\noindent  2020 AMS Classification: 32J27, 32C38, 32Q30, 14C30, 46L10, 58J22. 
 
\noindent Keywords: Complex manifolds, $\mathcal{D}$-modules, Constructible sheaves, Hodge Modules, Mixed Hodge Theory,   Atiyah's $L_2$-index theorem, Group Von Neumann algebras, $L_2$ Betti numbers. 
\end{abstract}

\maketitle

This article is an extension of Dingoyan's $L_2$-Mixed Hodge theory \cite{Din} and 
a first step towards a  version of Gromov's influential article on K\"ahler-hyperbolic manifolds \cite{Gro2} that would apply to singular K\"ahler varieties. 

Gromov found a way to use  the  $L_2$-De Rham theory of an infinite Galois covering space of a compact K\"ahler manifold $X$ and obtain algebro-geometric restrictions 
if $X$ is K\"ahler-hyperbolic, for instance a compact complex submanifold of a neat quotient of a bounded symmetric domain.
This inspired the influential works of Campana \cite{Cam,Cam2} and Koll\'ar \cite{Kol1,Kol2} masterfully exploiting  in K\"ahler geometry the striking ideas of \cite{Ati} to study compact K\"ahler manifolds with infinite fundamental group. 
Gromov's ideas  were also extended in \cite{Eys1} to polarized Variations of Hodge Structures 
(actually to harmonic bundles) on a {\em compact} K\"ahler manifold $X$.  
They were also extended in \cite{CamDem, Eys} to  a theory of coherent $L_2$-cohomology in Complex Analytic Geometry. 
Some applications 
were given, say in \cite{E5,E3, Taka},  and an extremely striking one was recently found  \cite{Braun}. 

With applications in mind, the author is interested in further extending 
the theory to Mixed Hodge Modules \cite{SaiMHM}. 

\vskip 0.3cm

Let $X$ be a compact complex manifold. Let $\pi: \widetilde{X} \to X$ be an infinite Galois covering space 
with $\mathrm{Deck}(\widetilde{X}/ X)=\Gamma$.

If $F^{\bullet}$ is a bounded complex of $\C$-vector spaces with constructible cohomology on $X$ we construct, using a classical observation of Kashiwara,   cohomology groups 
$\mathbb{H}_{(2)}^{\bullet} (\widetilde X , F^{\bullet})$ that co\"incide in the case  $F^{\bullet}= \C_X$ to the $L_2$-cohomology of $\widetilde{X}$, see \cite{Luc}. 
They obey Atiyah's $L_2$ index theorem,  Poincar\'e-Verdier duality and are compatible with proper morphisms of complex analytic spaces. 

If $\mathcal{M}$ is a coherent $\mathcal{D}$-module on $X$, 
we construct, using the construction of \cite{Eys},   cohomology groups 
$\mathbb{H}_{DR, (2)}^{\bullet} (\widetilde X , \mathcal{M})$ that co\"incide in the case  $\mathcal{M}= \mathcal{O}_X$   to the $L_2$-De Rham cohomology of $\widetilde{X}$ 
with respect to a Riemannian metric pulled back fom $X$ and if $\mathcal{M}=\mathcal{D}\otimes_{\mathcal{O}_X} \mathcal{F}$, $\mathcal{F}$ being a coherent analytic sheaf with the $L_2$ cohomology groups 
$H^{\bullet}_2(\widetilde{X}, \mathcal{F}\otimes \omega^n_X)$ constructed in \cite{Eys}. They obey Atiyah's $L_2$ index theorem. We did not check except in the simplest cases whether they are compatible with proper holomorphic mappings
and did not study Verdier duality.

When an isomorphism  in the derived category of sheaves $rh: F^{\bullet} \cong DR(\mathcal{M})$ is given, $\mathcal{M}$ being holonomic, we construct a natural  isomorphism
$rh_{(2)}: \mathbb{H}_{(2)}^{\bullet} (\widetilde X , F^{\bullet}) \cong \mathbb{H}_{DR, (2)}^{\bullet} (\widetilde X , \mathcal{M})$. 

These cohomology groups are typically infinite dimensional quotients of Hilbertian $\Gamma$-modules by non necessarily closed submodules. They are also modules over $\Ng$ the Von Neumann algebra of $\Gamma$.   But one can be much more precise. 

Given $\Gamma$ a discrete countable group,  the exact  category of finite type projective Hibert $\Gamma$-modules naturally embeds in a rather simple abelian 
category $E_f(\Gamma)$ due to  Farber \cite{Farb1} and L\"uck \cite{Luc}  endowed with a faithful functor to $Mod(\Ng)$. This abelian category has projective dimension one, its projective objects being finite type projective Hilbert $\Gamma$-modules. The preceding $L_2$ cohomology groups 
are in the essential image 
of the forgetful functor and the  isomorphism $rh_{(2)}$ lifts too.

\begin{theoi}\label{thm1}

Let $X$ be a compact complex manifold and $\widetilde{X} \to X$ be a Galois covering with Galois group $\Gamma$. 
Let $MD(X)$  be the abelian category  whose objects are triples $$\mathbb{M}=(\mathcal{M}=\mathbb{M}^{DR}, P=\mathbb{M}^{Betti}, \alpha)$$ where $\mathcal{M}$ is a holonomic $\Dx$-module  admitting a good filtration,
 $P$ is a perverse sheaf of $\R$-vector spaces and $\alpha: P\otimes_{\R}\C \to DR(\mathcal{M})$ is an isomorphism in the derived category of sheaves and whose morphisms are the obvious ones.

There is a  
 $\partial$-functor which, on the Betti side, is compatible with proper direct images, satisfies Atiyah's $L_2$ index theorem and Poincar\'e-Verdier duality:
$$ L_2dR: D^b MD(X) \to D^b E_f(\Gamma)
$$
and for each $\mathbb{M}\in MD(X)$ and $q\in \Z$  functorial isomorphisms in $E_f(\Gamma)$ $$H^q(L_2dR(\mathbb{M})) \cong\mathbb{H}^q_{(2)}(\widetilde{X}, \mathbb{M}^{Betti})\cong\mathbb{H}^q_{DR, (2)}(\widetilde{X}, \mathbb{M}^{DR}).$$ 
\end{theoi}

If $X$ is a
projective algebraic manifold every coherent $\Dx$-module admits a global good filtration   - a fact the author has learned from talks given by B. Malgrange. The author does not believe admitting a good filtration is an essential restriction here.

For applications, it seems to be useful to consider the case  $X$ is only a compact complex-analytic space such that one can embed $X$ in a complex manifold $Z'$. In that situation,  
one can construct, taking a regular neighborhood $Z$ of $X$,    an infinite Galois covering space 
with $\mathrm{Deck}(\widetilde{Z}/ Z)=\Gamma$ and a $\Gamma$-equivariant embedding $\widetilde{X} \to \widetilde{Z}$ covering the closed embedding $X \to Z$. Theorem  \ref{thm1} extends
to this situation restricting one's attention to modules on $Z$ whose support is contained in  $X$. 

Saito's category of Mixed Hodge Modules  $MHM(X)$ \cite{SaiMHM} is an abelian subcategory of $MD(X)$.

\begin{coroi}

Let $X$ be a compact K\"ahler manifold and $\widetilde{X} \to X$ be a Galois covering with Galois group $\Gamma$. 

There is a  
 $\partial$-functor which, on the Betti side, is compatible with proper direct images, satisfies Atiyah's $L_2$ index theorem and Poincar\'e-Verdier duality:
$$ L_2dR: D^b MHM(X) \to D^b E_f(\Gamma)
$$
and for each $\mathbb{M}\in MHM(X)$ and $q\in \Z$  functorial isomorphisms in $E_f(\Gamma)$ $$H^q(L_2dR(\mathbb{M})) \cong\mathbb{H}^q_{(2)}(\widetilde{X}, \mathbb{M}^{Betti})\cong\mathbb{H}^q_{DR, (2)}(\widetilde{X}, \mathbb{M}^{DR}).$$ 

These cohomology groups are endowed with  a
real structure, a real filtration $W$ coming from the weight filtration on $\mathbb{M}^{Betti}$ and a  complex filtration $F$  coming from Saito's Hodge filtration on $\mathbb{M}^{DR}$. These filtrations and real structures are compatible
with morphisms of Mixed Hodge Modules.

\end{coroi}


It is not clear whether these filtrations define a Mixed Hodge Structure. It seems difficult not to pass to reduced $L^2$-cohomology. 
 Using an idea of Dingoyan   \cite{Din}, we conjecture: 

\begin{conji}\label{theconj}
 Let $\Ng \subset \Ug$ be the algebra of affiliated operators. Let $\mathbb{M}$ be a Mixed Hodge module (resp. a pure Hodge Module).  
 $$\Ug \otimes_{\Ng} H^q(\wt{X}, L_2dR(\mathbb{M}))$$ carries a Mixed (resp. a pure) Hodge structure in the abelian category of real $\Ug$-modules with finite $\Gamma$-dimension. The restrictions on the 
 Hodge numbers are as in the compact case. 
\end{conji}

We will use the notation $H^q(L_2dR(\mathbb{M}))\buildrel{not.}\over{=} H^q(\wt{X}, L_2dR(\mathbb{M}))$ whenever it is necessary to emphasize that $\mathbb{M}$ lives on $X$ and that we are considering the covering space $\wt{X}\to X$.

We do not understand 
Saito's theory well enough to conjecture a similar statement for the derived category of $MHM(X)$. 

\begin{theoi}\label{thm2}
Conjecture \ref{theconj} is true in the following cases:
\begin{itemize}
 \item  There is a closed complex submanfold $i:Z\hookrightarrow X$ and a smooth polarized $\Q$-VSH $(Z, \mathbb V, F, S)$ on $Z$ such that $\mathbb{M}=\mathbb{M}_{i}(\mathbb{V})$ is the corresponding Hodge Module on $X$.
 \item  There is an open embedding $j:U\hookrightarrow X$ such that $X\setminus U$ is a divisor with 
 simple normal crossings and a smooth $\Q$-VSH $(X, \mathbb V, F, S)$ on $X$ such that $\mathbb{M}=Rj_*j^{-1} \mathbb{M}_{X}(\mathbb{V})$.
 \item  There is an open embedding $j:U\hookrightarrow X$ such that $X\setminus U$ is a divisor with simple normal 
 crossings and a smooth $\Q$-VSH $(X, \mathbb V, F, S)$ on $X$ such that $\mathbb{M}=Rj_!j^{-1} \mathbb{M}_{X}(\mathbb{V})$.
\end{itemize}

\end{theoi}
The first case follows easily from \cite{Eys1}.
The second item in case $\V=\Q_X$ follows from \cite{Din} and  Theorem \ref{thm2} could be proved with a slight variation on Dingoyan's approach. We nevertheless felt it was helpful  to recast Dingoyan's results in our language. 
The third case does not follow from \cite{Din}. A more general result holds, it is enough that the $Gr_W$ of the Mixed Hodge module is a direct sum of modules of the form $\mathbb{M}_{i}(\mathbb{V}^{\alpha})$.

The general case requires only to be able to settle the case of  pure polarizable Hodge modules.
The author believes one can settle this in case $\dim(X)=1$.  The author hopes the general case will be doable when a proof of the coincidence of the algebraic and the analytic definition of the Hodge filtration will be available. 

The author believes one can endow the reduced $L^2$-cohomology of an infinite Galois cover of a projective algebraic variety with a functorial $\Ug$-Mixed Hodge 
structure using  techniques developped here and cohomological descent \cite{DelH3} and hopes to come back to this question in a future work. Subsuming Deligne's approach into Saito's is not straightforward \cite{SaiMHC}
and one really needs to exercise some more care to dare  draw this conclusion. 

The recent preprint \cite{DinSch} suggests an extension of the theory for twistor $\Dx$-modules might be possible.

\vskip 0.2cm

The article is organized as follows. The first section constructs $L_p$-constructible cohomology. The second section constructs $L_p$ De Rham cohomology for coherent $\mathcal{D}$-module on complex manfolds. The third section reviews 
some facts on the homological algebra for $\Ng$-modules and about $\Ug$. The fourth section lifts the $L_2$-cohomology theory to $E_f(\Gamma)$ and  finishes the proof of Theorem \ref{thm1}. 
It gives a statement of a refined form of
 Conjecture \ref{theconj} in terms of the reduced $L^2$ cohomology of Mixed Hodge Modules and a brief treatment of the singular case. The fifth section
studies analytic $L^2$ Hodge decomposition in the K\"ahler case. The sixth section gives a proof of Theorem \ref{thm2}. An appendix gives more details on  some technical facts the author prefered not to include in the text. 
A final section indicates briefly how using the algebra of affiliated operators simplifies the theory in \cite{Eys}
and gives a version of Theorem \ref{thm1} valid without good filtrations. The reader is referred
to \cite{Eys1,Din} for examples.

The article is a write-up of a project that was started 18 years ago with the definition of constructible  $L_2$-cohomology in $E_f(\Gamma)$. 
After \cite{Din} appeared,  the scope of the project was extended to include Mixed Hodge Modules.
The author has  given a handful of seminar and conference
talks on this project during these years and wishes to apologize for not having made a text available. At some point, it was a work in collaboration 
with P. Dingoyan, who withdrew from the projet. The author would like to address special thanks to him for 
many enlightening discussions. 

The author also thanks  P. Bressler, S. Diverio, S. Guillermou, F. Ivorra, B. Jean, W. L\"uck, J. Tapia and C. Sabbah for valuable discussions, some of them 20 years ago, on topics related to this article.

\section{\label{sec1} Constructible $L_p$-cohomology}

In the following $1\le p <+\infty$ will be a real number. No applicable results will be lost if one restricts oneself to the case $p=2$. We also let $\Q\subset K \subset \C$ be  a subfield of the complex numbers.

\subsection{\label{sc} Equivariant constructible sheaves on $\Gamma$-simplicial complexes}

In this section, we recall basic well-known definitions, cf \cite{KS}, chap. VIII.

Let $\Gamma$ be a discrete countable group.
Let $T$ be a paracompact topological space endowed with an action
of $\Gamma$ (by homeomorphisms). 
We denote
by  $\mathbf{Mod}_{\Gamma}(K_T)$,  the category of $\Gamma$-equivariant sheaves of $K$-vector spaces
\footnote{ A compatible action of $\Gamma$ on a sheaf $\mathcal S$ is a continuous action on
$Et(\mathcal{S})$ the espace \'etal\'e of $\mathcal{S}$ such that the canonical
local homeomorphism $Et(\mathcal{S})\to T$ is $\Gamma$-equivariant. }.
 Let  $A$ be an abelian category, we also call $D^b(A)$ its bounded derived category\footnote{There is no need to restrict to the bounded derived category until section 4.2, but we will not pursue more generality.}.
We use the shorter notation $D^b_{K,\Gamma}(T) :=D^b \mathbf{Mod}_{\Gamma}(K_T)$. 
We shall drop dependance on $K$ when $K=\C$.

A  $\Gamma$-simplicial complex $\mathbb S$
is a locally finite
simplicial complex endowed with a proper left action
of $\Gamma$, i.e. $\mathbb{S}=(S,\Delta, i)$ where $S$
is a non-empty set endowed with an action of
$\Gamma$ $i:\Gamma\to \mathfrak{S}(S)$ and $\Delta$
 is a set of non-empty finite subsets of $S$, the {\em simplices}
  of $\mathbb{S}$
such that:

\begin{itemize}
\item For every element $s$ of $S$, the singleton $\{ s \}$  belongs to $\Delta$.
\item For every element $\sigma$ of $\Delta$, any non-empty subset $\tau$ of $\sigma$
belongs to $\Delta$.
\item For every element $s$ of $S$, the subset of $\Delta$ consisting in the simplices containing
$s$ is finite.
\item $\Gamma$ preserves $\Delta$.
\item $\Gamma$ acts on $S$ with finite stabilizers.
\end{itemize}

Obviously, $\Gamma$ acts on $\Delta$ with finite stabilizers and $\Gamma$ acts properly
 on the topological realization $|\mathbb{S} |$
of $\mathbb{S}$. $|\mathbb{S}|$ is a closed subspace of $\mathbb{R}^S$
(endowed with the product topology) decomposed as $|\mathbb{S}|=\cup_{\sigma\in\Delta} |\sigma|$
where
 $$|\sigma|=\{ x\in \mathbb{R}^{S} | x(p)=0  \ {\rm{if}}  \ p\not \in \sigma, \  x(p)>0 \ {\rm{if}}  \ p \in \sigma, \
\sum_{p\in\sigma} x(p)=1\}.$$

Say $\mathbb{S}$ is finite dimensional if $\sup_{\sigma\in\Delta} {\text{Card}} (\sigma)<\infty$.
Say $\mathbb{S}$ is cocompact if it is finite dimensional and $\Gamma\backslash S$ is finite.
In this case,  $\Gamma \backslash \Delta$ is finite and
$\Gamma \backslash |\mathbb{S}|$ is compact.

A $\Gamma$-equivariant sheaf of $K$-vector spaces ${F}$ on
$|\mathbb{S}|$ is {\em weakly} $\mathbb{S}$-{\em constructible} sheaf,
resp. $\mathbb{S}$-{\em constructible}, if for
 every simplex $\sigma$, $i_{|\sigma|}^{-1} F$ is a constant sheaf\footnote{
Whenever $Z$ is a locally closed subset of $X$, we denote by $i_Z:Z\to X$
the resulting embedding.},
resp. and for every $x$ in $|\mathbb{S}|$, $F_x$ is of finite dimension. The abelian category of
$\mathbb{S}$-constructible (resp. weakly $\mathbb{S}$-constructible) equivariant  sheaves
will be denoted by $\mathbf{Cons}_{K,\Gamma}(\mathbb{S})$ (resp. $w\mathbf{Cons}_{K,\Gamma}(\mathbb{S})$)
A complex of $\Gamma$-equivariant  sheaves $F^{\bullet}$ with bounded
cohomology (i.e.: an object of $D^b_{\Gamma}(|\mathbb{S}|))$
is called  $\mathbb{S}$-constructible (resp. weakly $\mathbb{S}$-constructible) 
if its cohomology sheaves $\mathcal{H}^j(F^{\bullet})$
are  $\mathbb{S}$-constructible (resp. weakly $\mathbb{S}$-constructible). $\mathbb {S}$-constructible complexes (resp. weakly
$\mathbb {S}$-constructible complexes) are the objects of a full thick triangulated subcategory
$D^b_{\mathbb{S}-c,K,\Gamma}(|\mathbb{S}|)$ (resp. $D^b_{w-\mathbb{S}-c,K,\Gamma}(|\mathbb{S}|)$) of $D^b_{K,\Gamma}(|\mathbb{S}|)$.

\begin{prop} \label{ks8111} Let $\mathbb{S}$ be a finite dimensional $\Gamma$-simplicial complex.
Then the natural functors

$$ D^b (w\mathbf{Cons}_{K,\Gamma}(\mathbb{S})) \to D^b_{w-\mathbb{S}-c,K, \Gamma}(|\mathbb{S}|), \ 
D^b (\mathbf{Cons}_{K, \Gamma}(\mathbb{S})) \to D^b_{\mathbb{S}-c,K, \Gamma}(|\mathbb{S}|)$$

are equivalences of triangulated categories.
\end{prop}
\begin{prf}
The proof of Theorems 8.1.10 and  8.1.11 p.326 in
\cite{KS} (which is the special case where $\Gamma$
is the trivial group) applies here mutatis mutandis.
Actually if the action of $\Gamma$ is free we can use the natural equivalence of categories between the various categories of
$\Gamma$-equivariant sheaves on $|\mathbb{S}|$  and of sheaves on $\Gamma \backslash |\mathbb{S}|$ to formally reduce the statement to \cite[Chapter VIII]{KS}. 
\end{prf}

\subsection{$L_p$-cohomology for equivariant constructible sheaves. The case of simplicial complexes.}

Let $\mathbb{S}$ be a finite dimensional $\Gamma$-simplicial complex.
Consider the natural quotient map $\pi: |\mathbb{S}| \to \Gamma \backslash |\mathbb{S}|$.
It is easy to see that $\pi_!$ the direct image with proper
support is exact on $\mathbb{S}$-constructible sheaves \footnote{It is a direct consequence of
 \cite{KS} Proposition 8.1.4  p. 323 in case the action is free. The general case is easily taken care
 of by a barycentric subdivision argument.}.
On the category of equivariant $\mathbb{S}$-constructible sheaves, $\pi_!$ factorizes through
the category of sheaves of left $K\Gamma$-modules on $\Gamma \backslash |\mathbb{S}|$.

The left and right regular representations, denoted by $\lambda$ and $\rho$, on  the set $l_p\Gamma$ of complex
valued functions $(a_{\gamma})_{\gamma \in \Gamma}$
defined on $\Gamma$ such that $\sum_{\gamma\in\Gamma} |a_{\gamma}|^p <\infty$ defines a bimodule
over $\C\Gamma$.
We call $Rl_p\Gamma$ the right $\Gamma$-module attached to $\rho$.
In particular, given a sheaf $F$ of left $K\Gamma$-modules on a topological
space $T$, the tensor product $Rl_p\Gamma \otimes_{K\Gamma} F$ is a sheaf of left
$\C  \Gamma$-modules. It is actually a sheaf of $W_{l,p}(\Gamma)$-modules where $W_{l,p}(\Gamma)$
 is the
bicommutant of  $\lambda (\mathbb{Z}\Gamma)$
in the algebra of continuous linear endomorphisms of $l_p\Gamma$.

\begin{lem} The functor $F\mapsto Rl_p\Gamma \otimes_{K\Gamma} \pi_!F$ is exact
on $\mathbf{Cons}_{K,\Gamma}(\mathbb{S})$.
\end{lem}

\begin{prf} Since $\mathbb C$ is a  $K$-vector space it is flat over $K$ and $\pi_!$ commutes
with $\otimes_{K}\mathbb C$.  Hence, it is enough to prove exactness of $F_{\mathbb C} \mapsto Rl_p\Gamma \otimes_{\mathbb{C}\Gamma} \pi_!F_{\mathbb C}$
on $\mathbb S$-constructible equivariant sheaves of $\mathbb C$-vector spaces. 
Since the stalk at $p\in \Gamma\backslash |\mathbb{S}|$
of $\pi_!\mathcal F_{\mathbb C}$ is isomorphic to $\mathbb{C} [\Gamma \slash H_{\tilde p}]^{\oplus n}$ where $n$ is a nonnegative integer and $H_{\tilde p}$ is the stabilizer of  some lift $\tilde p \in |\mathbb{S}|$
of $p$, it follows that it is a projective module over $\mathbb{C}\Gamma$.
Indeed, whenever $H$ is a finite subgroup of $\Gamma$,
 the pull-back injection $i:\mathbb{C} [\Gamma/H] \to \mathbb{C}\Gamma$ has a right inverse $\pi((a_{\gamma})_{{\gamma}\in \Gamma})_g=
\frac{1}{\text{Card}(H)} \sum_{h\in H} a_{gh}$ which is equivariant
for the left action of $\Gamma$.
Exactness follows from the facts that stalks of tensor products are computed stalkwise, that $\pi_!$ is exact and
 that a short exact sequence of projective modules splits.
\end{prf}
\begin{lem} \label{extension1} The functor $F\mapsto Rl_p\Gamma \otimes_{K\Gamma} \pi_!F$ is exact
on $\mathbf{Mod}_{\Gamma}(K_{|\mathbb{S}|})$.
\end{lem}
\begin{prf} The above proof also works with a minor modification for all sheaves of $K$-vector spaces. 
\end{prf}

\begin{defi} \label{lpco}
Let $F^{\bullet}$ be an object of $D^b(\mathbf{Cons}_{K,\Gamma}(\mathbb{S}))$. Its $k$th $L_p$- hypercohomology
groups is the $W_{l,p}(\Gamma)$-module
$$\mathbb{H}^k_{(p)}(|\mathbb{S}|, F^{\bullet}):=
\mathbb{H}^k(\Gamma \backslash |\mathbb{S}|,
Rl_p\Gamma\otimes_{K\Gamma} \pi_!
F^{\bullet}).$$
\end{defi}
\begin{lem}
The composition of the derived functor of $H^0(|\mathbb{S}|, -)$ and of $Rl_p\Gamma \otimes_{K\Gamma} -$ gives a $\partial$-functor  \footnote{As in \cite[p. 22]{HaRD}, a $\partial$-functor of triangulated categories is an additive functor 
which commutes with the translation functor and respects distinguished triangles. } of triangulated categories
 $$\mathbb{H}^\bullet_{(p)}(|\mathbb{S}|,-):  D^b(\mathbf{Cons}_{K,\Gamma}(\mathbb{S}))\to D^b ({\text{Mod}}_{W_{l,p}(\Gamma)}),$$
where
${\text{Mod}}_{W_{l,p}(\Gamma)}$ stands for the
category of left $W_{l,p}(\Gamma)$-modules such that the $L_p$-hypercohomology  groups are its cohomology objects.
\end{lem}

This definition gives rise to the long
exact sequence attached to a short exact sequence
and to various spectral sequences generalizing it.

Thanks to lemma \ref{extension1}, we also get:

\begin{lem} \label{lemext1}
The same formula as in definition \ref{lpco} defines an extension of $\mathbb{H}^*_{(p)}(|\mathbb{S}|,-)$ to a $\partial$-functor  $$\mathbb{H}^\bullet_{(p)}(|\mathbb{S}|,-):D^b_{K,\Gamma}(|\mathbb{S}|) \to  D^b ({\text{Mod}}_{W_{l,p}(\Gamma)}).$$
\end{lem}

We don't use a different notation hoping this will not cause any confusion. 

\subsection{The subanalytic case}

\subsubsection{Subanalytic stratifications and constructible sheaves}



A subanalytic\footnote{Actually, \lq  definable in a o-minimal structure\rq \  is the natural hypothesis.} $\Gamma$-space is a $\Gamma$-space  that can be 
realized as a locally closed $\Gamma$-invariant subanalytic
subset of a real analytic 
manifold 
endowed with a proper real analytic action of $\Gamma$. 
A stratified subanalytic space $\mathbb{X}$ is
a subanalytic space $\wt{X}:=|\mathbb X|$ endowed with a locally finite partition
$\wt{X}=\cup_i X_i$ in disjoint subanalytic submanifolds satisfying
$X_i\cap \bar X _j\not = \emptyset \Longrightarrow X_i\subset \bar X_j$.
A stratified subanalytic $\Gamma$-space $\mathbb X$ is a proper
analytic action of $\Gamma$ on $\wt{X}=|\mathbb{X} |$  such that
for every $g\in \Gamma$ and every point $x\in |\mathbb{X} |$ the germ
 at $x$ of the stratification is carried by $g$ to the germ at $gx$ of
 the stratification. When the stratification comes from a $\Gamma$-simplicial complex, 
one calls it a triangulation.
Proposition 8.2.5 in \cite{KS}
implies that, in the cocompact case, any $\Gamma$-stratification may be refined
to a $\Gamma$-triangulation and that the $\Gamma$-triangulations form a  cofinal system with respect to refinement.

The obvious extension of the definitions and notations of section \ref{sc}
will be left to the reader, the only change being that $\mathbb{X}$-constructible sheaves on $|\mathbb{X}|$
are now assumed to be locally constant along the strata of $\mathbb{X}$.

Let $\wt{X}$ be a subanalytic $\Gamma$-space. A sheaf of $K$-vector spaces $F$ on $\wt{X}$
is called constructible if it is constructible with respect to
some subanalytic stratification of $\wt{X}$.
We denote by $ \R \mathbf{Cons}_{K,\Gamma}(\wt{X})$  the category
of equivariant constructible sheaves on $\wt{X}$
and by  $D^b_{K, \R c,\Gamma}(\wt{X})$ the thick full subcategory of $D^b_{K, \Gamma}(\wt{X})$
consisting of complexes with bounded constructible cohomology.
Theorem 8.4.5 in \cite[p. 339]{KS} is easily generalized to
\begin{prop}
 The natural functor $D^b (\mathbb{R} \mathbf{Cons}_{K,\Gamma}(\wt{X}) )
\to D^b_{\mathbb{R}c, K, \Gamma}(\wt{X})$ is an equivalence of triangulated categories if $\wt{X}$ is cocompact.
\end{prop}

\subsubsection{Constructible $L_p$-cohomology}

\begin{prop}\label{lpcocons}
Let $\wt{X}$ be a cocompact subanalytic $\Gamma$-space.
Let $\mathcal{F}^{\bullet}$ be an object of $D^b_{\R-c,K, \Gamma}(\wt{X})$. Its $k$th $L_p$-hypercohomology
groups is the $W_{l,p(\Gamma)}$-module
$$\mathbb{H}^k_{(p)}(\wt{X}, \mathcal{F}^{\bullet}):=
\mathbb{H}^k(\Gamma \backslash \wt{X},
Rl_p\Gamma\otimes_{K\Gamma} \pi_!
\mathcal{F}^{\bullet}).$$

There is a $\partial$-functor of triangulated categories $$\mathbb{H}^\bullet_{(p)}(\wt{X},-): 
 D^b_{\R-c, K,\Gamma}(\wt{X})\to D^b ({\text{Mod}}_{W_{l,p}(\Gamma)}), $$ where
${\text{Mod}}_{W_{l,p}(\Gamma)}$ stands for the
category of left $W_{l,p}(\Gamma)$-modules such that
$$\mathbb{H}^k_{(p)}(\wt{X}, \mathcal{F}^{\bullet})= H^k(\mathbb{H}^k_{(p)}(\wt{X}, \mathcal{F}^{\bullet})).$$

\end{prop}
\begin{prf}
We can replace $D^b_{\R-c,K,\Gamma}(\wt{X})$ by $D^b\R\mathbf{Cons}_{K,\Gamma}(\wt{X})$
since the natural functor is an equivalence by Proposition \ref{ks8111} and $Rl_p\Gamma\otimes_{K\Gamma} \pi_!  \C \otimes -$ is exact by Lemma \ref{lemext1}. 

Since $D^b\R\mathbf{Cons}_{K,\Gamma}(\wt{X})$ is the limit of its full subcategories
 $D^b\mathbf{Cons}_{K,\Gamma}(\mathbb{X})$, $\mathbb{X}$
running through all subanalytic $\Gamma$-triangulations,
this follows from definition \ref{lpco}.

\end{prf}

\subsection{Complex Analytic case}

We assume here $K=\C$ and drop $K$ from the notation. 

Assume from now on that $\wt{X}$ is a cocompact complex $\Gamma$-space. The relevant stratifications are complex analytic  stratification
(by definition, a subanalytic stratification
is complex analytic if so are the closures of the strata)
and we say that an equivariant sheaf is constructible if is constructible with respect to some complex analytic
stratification and that a complex of equivariant sheaves is constructible if so are its cohomology sheaves. Then  $\mathbf{Cons}_{\Gamma}(\wt{X})$ is a full 
abelian subcategory of 
$\mathbb{R}\mathbf{Cons}_{\Gamma}(\wt{X}) $ stable by extensions, 
$D^b_{c,\Gamma}(\wt{X})$, the full  subcategory of $D^b_{\mathbb{R}c,\Gamma}(\wt{X})$ whose cohomology objects are in $
\mathbf{Cons}_{\Gamma}(\wt{X})$, is a thick triangulated subcategory
and we have a natural $\partial$-functor
$$D^b\mathbf{Cons}_{\Gamma}(\wt{X}) 
\to D^b_{c,\Gamma}(\wt{X}).$$
\begin{rem}
This functor is an equivalence of categories if $\wt{X}$ is a Galois topological covering space of the analytization of a  complex projective 
variety using GAGA and \cite{Nori}. 
\end{rem}


\begin{defi}\label{CconstrL2}
We can restrict  $\mathbb{H}^\bullet_{(p)}(\wt{X},-)$
to $D^b_{c,\Gamma}(\wt{X}) $ to get the {\em constructible $L_p$-cohomology functor}:
$$\mathbb{H}^\bullet_{(p)}(\wt{X},-):D^b_{c,\Gamma}(\wt{X})\to  D^b ({\text{Mod}}_{W_{l,p}(\Gamma)})
.$$
\end{defi}

An important special case is $L_p$-intersection cohomology. 
\begin{defi}\label{l2ic}
 Let $Z$ be a singular compact complex space and $\pi:\wt{Z}\to Z$ its universal covering space. Its $k$-th intersection
$L_p$ cohomology is the $W_{l,p}(\Gamma)$-module
$\mathbb{H}^k_{(p)}(\widetilde{Z}, \pi^{-1} \mathcal{IC}^{\bullet}_Z)$
where $\mathcal{IC}^{\bullet}_Z$ is the intersection cohomology sheaf of $Z$
\cite{BBD}.
\end{defi}

The initial impetus for this work was to formulate the following:

\begin{conj}
 If $p=2$, and $Z$ is a closed analytic subset of a compact K\"ahler hyperbolic manifold then $\mathbb{H}^k_{(2)}(\widetilde{Z}, \pi^{-1} \mathcal{IC}^{\bullet}_Z)=0$ for $k\not=\dim(Z)$.
\end{conj}

\subsection{Real structures}

The algebra $W_{l,p}(\Gamma)$ carry a real structure, namely a conjugate 
linear algebra involutive automorphism $\dagger$,   and a real structure on a module 
over $W_{l,p}(\Gamma)$ is just a conjugate linear automorphism on the underlying $\C$-vector space compatible with $\dagger$. For instance $l_p \Gamma$ has a real structure.

Modules with real structures form a $\R$-linear abelian category  $Mod_{W_{l,p}(\Gamma),\R}$ which has an exact faithful forgetful functor  to $Mod_{W_{l,p}(\Gamma),\R}$ and if $K\subset \R$ the functors of Proposition \ref{lpcocons}
and Definition \ref{CconstrL2} lift to 
$D^b(Mod_{W_{l,p}(\Gamma),\R})$. 

\section{\label{sec2}$L^p$-cohomology and differential operators for coherent sheaves}

\subsection{Coherent $L^p$-cohomology}

Let $X$ be a complex manifold and let $\Ox$ (resp. $\Dx$) denote its structure sheaf (resp. the sheaf 
of holomorphic differential operators). Let $\pi:\wt{X} \to X$ be a Galois topological covering space and 
$\Gamma=\mathrm{Gal}(\wt{X}/X)$ be its Galois group, acting on $\wt{X}$ on the left.

If $\mathcal{R}$ is a sheaf of rings on $X$, denote by $\mathbf{Mod}(\mathcal{R})$
the abelian category of sheaves of left $\mathcal{R}$-modules by $Hom_{\mathcal{R}}$ its group of morphisms and by $\underline{Hom}_{\mathcal{R}}$
the  internal Hom bifunctor on $\mathbf{Mod}(\mathcal{R})$. When considering
right $\mathcal{R}$-modules, we use the notations $\mathbf{Mod}(\mathcal{R}^o)$, $Hom_{\mathcal{R}^o}$, $\underline{Hom}_{\mathcal{R}^o}$. 
If $R$ is a ring then we denote by $R_X$  the sheaf of rings 
of locally constant functions with values in $R$. 

Denote  by $\coh$ the full abelian  subcategory of  $\omod$
whose objects are the
coherent analytic sheaves of $X$. 
Denote  by $\cohd$ (resp. $\hol$) its the full abelian
subcategory of $\dmod$
whose objects are the
coherent (resp. holonomic)  $\Dx$-modules. A $\Ox$-module is quasi coherent if it is locally the limit of its 
coherent submodules. 

\paragraph{}
For every $p\in[1,+\infty[$ and $\F$ a coherent analytic sheaf on $X$, 
\cite{Eys} (see also \cite{CamDem}) constructs
a subsheaf $l^p \pi_* \F \subset \pi_*\pi^{-1} \F$   
which can be described locally as follows. Choose $\varphi:\Ox^{\oplus N}\to \F$
a presentation of $\F$ on  a Stein open subset $U$ such that $\pi^{-1}(U)=\Gamma \times U$ then: 
\begin{eqnarray*}
l^p \pi_* \F (U)&=&\{ (s_{\gamma})_ {\gamma \in \Gamma} \in \F(U)^{\Gamma}, \  
 \exists s_{\gamma}' \in \Ox^{\oplus N} (U)^{\Gamma},  \ 
\varphi(s'_{\gamma})=s_{\gamma} \  \mathrm{and} \\
 &&\forall K\Subset U  \quad  \sum_{\gamma\in \Gamma} \int_{K} |s_{\gamma}|^p <+\infty \}.
\end{eqnarray*} 
The independance on $\varphi$ is checked in \cite{Eys}. 
Given $\phi: \F \to \F'$ a $\Ox$-linear morphism of coherent sheaves 
$$\pi_*\pi^{-1}\phi: \pi_*\pi^{-1} \F \to \pi_*\pi^{-1} \F'$$ maps $l^p \pi_* \F$ into $l^p \pi_* \F'$. 
Denote by $l^p\pi_*\phi: l^p \pi_* \F \to l^p \pi_* \F'$ the restriction of $\pi_*\pi^{-1} \phi$. 
The resulting functor $l^p\pi_*: \coh \to \mathbf{Mod}(W_{l,p}(\Gamma) \otimes_{\C} \Ox)$ is  exact \cite{Eys} and 
one can define  $$H_{L^p}^{\bullet}(\wt{X}, \F):= H^{\bullet}(X, l^p\pi_*\F).$$ 

Since $H^q_{L^p}(\wt{X}, \F)=0$ for $q>\dim_{\C}(X)$ (at least when $X$ is compact) this yields a 
good cohomology theory  on $\coh$, indeed a 
a $\partial$-functor
$$D^b \mathbf{Coh}(\Ox) \to D^b Mod_{W_{l,p}(\Gamma)}.$$

Observe that if $\F$ is coherent $l^p\pi_* \F= l^p\pi_*\Ox \otimes_{\Ox} \F$. 
Hence, the functor  $l^p\pi_*$  extends   to $D^b (\omod)$ setting $l^p\pi_* \L:= l^p\pi_*\Ox \otimes_{\Ox} \L$ thanks to:

\begin{lem}
 The functor $l^p\pi_* = l^p\pi_* \Ox \otimes_{\Ox}$ is exact on $\mathbf{Mod}(\Ox)$. 
\end{lem}

\begin{prf}
 The problem is local. Since this functor is exact on $\mathbf{Coh}(\Ox)$,  it follows  from the fact that tensor products of sheaves commute with 
 taking the stalks \cite[p. 137]{God} that $Tor_1^{\mathcal{O}_{X,x}}( (l^p\pi_* \Ox)_x, \mathcal{O}_{X,x}/I_x)=0$ for every (finitely generated) ideal of
  $\mathcal{O}_{X,x}$. Hence $(l^p\pi_* \Ox)_x $ is a flat $\mathcal{O}_{X,x}$-module and exactness follows applying \cite[p. 137]{God} once more. 
\end{prf}

\paragraph{}

It should however be noted that if a sheaf $\F$ has two different  $\Ox$-module structures,
say $\F_1$ and $\F_2$, it may be the case that $l^p\pi_*\F_1\neq l^p\pi_* \F_2$ as subsheaves of
$\pi_*\pi^{-1} \F_1=\pi_*\pi^{-1} \F_2=\pi_*\pi^{-1} \F$. 

\begin{rem}\label{RemLpdmo}
There is a natural structure of left $\Dx$-module on $l^p\pi_* \Ox$ hence $l^p\pi_*$ gives rise to an exact endofunctor of $\mathbf{Mod}(\Dx)$. When $\Gamma$ is infinite, 
it does not preserve the full subcategory of coherent or quasicoherent modules. 
\end{rem}

\subsection{Differential operators}
We need to check that differential operators between quasicoherent analytic sheaves preserve $l^p\pi_*$. 

Recall from \cite{Sai} that for $\L, \L'$ two $\Ox$-modules $\diff (\L,\L')$ is the image of
the natural injective morphism:
$$Hom_{\Dx^o}(\L\otimes_{\Ox} \Dx, \L'\otimes_{\Ox} \Dx)\to Hom_{\C_X}(\L,\L')$$
given by the composition of the natural adjunction $$Hom_{\Dx^o}(\L\otimes_{\Ox} \Dx, \L'\otimes_{\Ox} \Dx)
\buildrel{\sim}\over\to Hom_{\Ox}(\L, \L'\otimes_{\Ox} \Dx),$$  $\L'\otimes_{\Ox} \Dx$ being endowed with the
{\bf{right}} $\Ox$-module structure, with left composition by the natural $\C$-linear (actually left $\Ox$-linear)
morphism
$$ \nu_{\L'}:  \L'\otimes_{\Ox} \Dx \to \L'
$$
which maps $\ell \otimes P$ to $P(1)\ell$. One has $\nu_{\L'}= \L'  \otimes_{\Ox} \nu_{\Ox} $
where $\nu_{\Ox}: \Dx\to \Ox$ is the naturel {\bf{left}} $\Dx$-linear
(hence left $\Ox$-linear) morphism mapping $P\in \Dx$ to $P(1)\in \Ox$.

\begin{lem}\label{cauchy} Assume $\L'$ is quasicoherent.  Let $(\L'\otimes_{\Ox} \Dx) _l$ resp. $(\L'\otimes_{\Ox} \Dx )_r$
the left resp. the right $\Ox$-modules structures of $\L'\otimes_{\Ox} \Dx$. Then:
$$
l^p \pi_* (\L'\otimes_{\Ox} \Dx) _l=l^p \pi_* (\L'\otimes_{\Ox} \Dx) _r \subset \pi_*\pi^{-1} \L'\otimes \Dx.
$$ 
 \end{lem}

\begin{prf}
 Let us begin by treating the case where $\L'=\Ox$. Then both $l^p(\pi_* (\Dx))_{\sharp}$,  $\sharp=l,r$ 
are the increasing union of the subsheaves $l^p \pi_* (F_k\Dx)_{\sharp}$ where $F_k \Dx$
is the sub-$\Ox$-bimodule consisting of the holomorphic differential operators of degree $\le k$. 
Hence it is enough to show that $l^p \pi_* (F_k\Dx)_l=l^p \pi_* (F_k\Dx)_r$.
The problem being local assume we have a coordinate system on an open set $U$ such that 
$\pi^{-1} (U)\simeq \Gamma \times U$. A section of $\pi_*\pi^{-1} F_k\Dx$ of the form 
$(\sum_{|\alpha|\le k} f_{\alpha,\gamma} \partial ^{\alpha} )$
is in $l^p(\pi_* (\Dx))_{l}$ iff, for all $K\Subset U$, $\sum_{\gamma}\int_K \sum_{\alpha} |f_{\alpha,\gamma}|^p <+\infty$
whereas a section $\pi_*\pi^{-1} F_k\Dx$ of the form 
$(\sum_{|\alpha|\le k}  \partial ^{\alpha} g_{\alpha,\gamma}) $
is in $l^p(\pi_* (\Dx))_{r}$ iff, for all $K\Subset U$,
 $$\sum_{\gamma}\int_K \sum_{\alpha} |g_{\alpha,\gamma}|^p <+\infty.$$
Since $\partial^{\alpha} g= g\partial^{\alpha} + \sum_{\beta < \alpha}P_{\beta,\alpha}(g) \partial^{\beta}$
where $P_{\beta,\alpha}$ is a universal differential operator, the Cauchy inequality gives
$$\int_K \sum_{\alpha} |f_{\alpha,\gamma}|^p  \le C_{K,K'} \int_{K'} \sum_{\alpha} |g_{\alpha,\gamma}|^p $$
if $K'\Subset U$ is a compact neighborhood of $K$. Whence
 the inclusion $l^p \pi_* (F_k\Dx)_l \subset l^p \pi_* (F_k\Dx)_r$. The reverse inclusion follows by the same token. 

This implies  the lemma  for $\L'$ a free $\Ox$-module of possibly infinite rank. 

Now, for the general case. 
The statement being local, we may 
choose $\varphi: \Ox^N \to \L'$ a presentation,  $N$ being some cardinal.  The definition implies that 
$l^p\pi_* (\L'\otimes_{\Ox} \Dx)_{\sharp} \in \pi_*\pi^{-1} \L'\otimes_{\Ox} \Dx$ is 
the image by $\pi_*\pi^{-1} \varphi$ of $l^p\pi_* (\Ox^N \otimes_{\Ox} \Dx)_{\sharp}$
 in $\pi_*\pi^{-1} \L'\otimes_{\Ox} \Dx$.  The lemma follows. 
\end{prf}

\begin{lem} Under the hypothesis of Lemma \ref{cauchy}, 
 let $P\in \diff(\L,\L')$ and let $p\in Hom_{\Ox}(\L, \L'\otimes_{\Ox} \Dx)$ be the unique
right $\Ox$-linear morphism such that $P=\nu_{\L'}\circ p$. Then $$l^p\pi_*P:= \nu_{l^p\pi_* \L'}\circ l^p\pi_* p:
l^p\pi_*\L \to l^p\pi_*\L'$$
is the restriction of $\pi_*\pi^{-1} P$ and defines a $W_{l,p}(\Gamma) \otimes_{\C}\Ox$-linear morphism of sheaves. 

The assignement $P\mapsto l^p\pi_* P$ defines an additive functor $$l^p\pi_*:\mathbf{Qcoh} (\Ox, \diff)\longrightarrow \mathbf{Mod}(\underline{W_{l,p}(\Gamma)}_X)$$
where $\underline{W_{l,p}(\Gamma)}_X$ is the constant sheaf with constant value $W_{l,p}(\Gamma)$ and  $\mathbf{Qcoh} (\Ox, \diff)$ is the additive category
whose objects are quasi coherent $\Ox$-modules and whose morphisms are differential operators. 
\end{lem}
\begin{prf} The definition makes sense thanks to lemma  \ref{cauchy}. The statement is thus an
 easy consequence of the definition and of the properties of $l^p\pi_*$ described above.  
\end{prf}

\subsection{$L_p$ De Rham cohomology}

Let $\M$ be a (quasi) coherent $\Dx$-module viewed as a $\Ox$-module endowed with a flat connection
$\nabla: \M \to \M \otimes_{\Ox} \Omega^1_X$. The De Rham complex of $\M$  defined as:
$$
DR(\M)=(\M \buildrel{\nabla}\over \to \M \otimes_{\Ox}\Omega^1_X \buildrel{\nabla}\over \to \M \otimes_{\Ox}\Omega^2_X \to \ldots )[\dim X]
$$
is a complex in $\mathbf{Qcoh} (\Ox, \diff)$. Applying the functor $l^p\pi_*$ we define the 
$L^p$ De Rham complex $l^p\pi_* DR(\M)$ and the $L^p$ De Rham cohomology:
$$ H^{\bullet}_{DR, L^p} (\tilde X, \M):= H^{\bullet}(X, l^p\pi_* DR(\M)).
$$
We will not try to put more structure than the natural $W_{l,p}(\Gamma)$-module structure on these general
 $L^p$ cohomology groups. 

The $L^p$ De Rham constructible cohomology groups come from a $\partial$-functor
$$H^{\bullet}_{DR, L^p,}: D^b(\cohd) \to D^b (Mod_{W_{l,p}(\Gamma)}). 
$$

\begin{exe}
 If $\F$ is a quasi coherent $\Ox$-module,  $$H^{\bullet}_{DR, L^p} (\wt{X}, \Dx\otimes_{\Ox} \F)=H^{\bullet}_{L^p} (\wt{X}, \F\otimes \omega_X).$$
\end{exe}

\begin{prf}
 The natural augmentation
 $\epsilon: \Dx \otimes_{\Ox}  \F\otimes_{\Ox}\omega_X  \to \F\otimes_{\Ox}\omega_X$
gives rise to a quasi-isomorphism $DR(\Dx\otimes_{\Ox} \F)\buildrel{\epsilon}\over\to \F\otimes_{\Ox}\omega_X$. 
Locally it is a Koszul complex for the regular sequence $(\partial_{x_1}, \ldots, \partial_{x_n})$.
The same is actually true for its $l^p \pi_*$
and  we get a quasi-isomorphism $l^p\pi_* DR(\F)\buildrel{l^p\pi_* \epsilon}\over\longrightarrow l^p\pi_* \F\otimes_{\Ox}\omega_X$.
\end{prf}

\begin{exe} Denote by $\ell^p\pi_* \C_{\wt{X}} \subset \pi_* \C_{\wt{X}}$ the locally constant  sheaf of $W_{l,p}(\Gamma)$-modules
 attached to the 
right regular representation of $\Gamma$ in $L^p\Gamma$. Let $\V$ be a finite rank complex local system
on $X$ and $\mathcal{V}$ be the $\Dx$-module whose underlying finite rank locally free
$\Ox$-module is $\V\otimes_{C_X}\Ox$ and holomorphic connection $\nabla$ so that the natural
morphism  $\sigma: \V \to \mathcal{V}$ represents $\ker(\nabla)$. 

Then $H^{\bullet}_{DR, L^p} (\wt{X}, \mathcal{V})= H^{\bullet+\dim(X)} (X, l^p\pi_* \C_{\wt{X}} \otimes_{\C_X} \V)
=\mathbb{H}^{\bullet+\dim(X)}_{(p)} (\tilde X, \V)$.
 
\end{exe}

\begin{prf} Left to the reader. 
\end{prf}

\begin{rem} With the notation of Remark \ref{RemLpdmo}, 
 $l^p\pi_* DR(\M)=DR(l^p\pi_* \M)$. 
\end{rem}

\subsection{Compatibility to the Riemann-Hilbert correspondance}

 \begin{rem}
 The natural sheaf monomorphism $i_{\pi}: \ell^p\pi_* \C_{\wt{X}} \otimes_{\C_X }\Ox \to l^p \pi_* \Ox$ is not an epimorphism. One would need a 
completed tensor product of sheaves in locally convex topological vector spaces we will not try and discuss.
\end{rem}

\begin{prop}\label{rhcomp}
 Let $\M$ be a holonomic $\Dx$-module. Then the natural map $Rl^p\Gamma \otimes_{\C\Gamma} \pi_!\pi^{-1} DR(\M) \to l^p\pi_*DR(\M)$
is a quasi-isomorphism. 
\end{prop}
\begin{prf}
The problem is local. Thus, we can assume $\M$ has a good filtration $F_.$. 
Then $F_.DR(\M)$ is a filtration of $DR(\M)$ by differential complexes of coherent sheaves. 
For $q \gg 0$ $F_q DR(M)\to DR(\M)$ is a quasi-isomorphism \cite[Lemma 1.14]{Sai}  (see also \cite[Lemma 1.5.6 p. 31]{Bjo}). Hence it is enough to
prove that $Rl^p\Gamma \otimes_{\C\Gamma} \pi_!\pi^{-1} F_qDR(\M) \to l^p\pi_*F_qDR(\M)$ 
is a quasi-isomorphism for such a $q\gg 0$. 

Thanks to the Kashiwara constructibility  theorem\footnote{Which is used implicitely in the statement of the proposition.} \cite{Ka,Me}, the cohomology of $F_qDR(\M)$ is constructible. 
Choosing $U$ appropriately such that it is Stein and  $$H^i(U,\mathcal{H}^j(DR(\M)))=0$$
for $i>0$. 
Note that the kernel and images
of  the differentials in  $F_q DR(\M)$ have also vanishing cohomology in positive degree on $U$. 

We have to show that 
every element $z$  of $Ker(d): F_{q+k}\M\otimes \Omega^{k} (U) \to F_{q+k+1}\M \otimes\Omega^{k+1}(U)$ 
can be decomposed as a sum $ z= d t+ g(h(z))$ where $$g: \mathcal{H}^{k-n}(F_q DR(\M)) \to \mathop{Ker}(d)\subset F_{q+k}\M$$
is a  section over $U$ of the morphism of sheaves
$$h: Ker(d: F_{q+k}\M\otimes \Omega^{k}  \to F_{q+k+1}\M \otimes \Omega^{k+1} ) \to \mathcal{H}^{k-n}(DR(\M))$$
and $t$ a section over $U$ of $F_{q+k-1}\M\otimes \Omega^{k-1}$ with local $L^p$-estimates. 

This means the following. The Fr\'echet structure of $\mathcal{G}(U)$ where $\mathcal{G}$ is coherent
 is given by an inverse limit of a countable family of $L^p$  norms $(\| -\|_n)_{n\in \N}$ defined by integration on an exhaustive family of compact subsets of $U$ if the sheaf 
is locally free,  of quotient norms of such  $L_p$ norms in a locally presentation of the sheaf in general \cite{Eys}. 
A local $L^p$ estimate is  then, for all $n \in\N$, a series of  estimates of the form:

$$ \| t \|_n \le C_n . \| z \|_{n'}
$$
for some $n'\in \N$.

This follows from  the continuity of $d$ for this Fr\'echet structure, the fact that a continuous
operator of Fr\'echet spaces has closed range if it has a finite dimensional kernel
and the open mapping theorem  for Fr\'echet spaces using  a standard argument (cf. e.g. \cite[pp. 534-535]{Eys}).  
\end{prf}

\begin{coro}
 The natural map $i_{\pi}$ induces a natural invertible transformation of functors on $\mathbf{Hol}(\Dx)$: 
 $$rh_{(p)}: H^{\bullet}_{DR, L^p} (\wt{X}, \_) \buildrel{\simeq}\over\longleftarrow \mathbb{H}^{\bullet+\dim_{\C}(\tilde{X}}_{(p)}(\wt{X}, DR(\_)).
 $$

\end{coro}

\begin{rem}
 As in section \ref{sec1} we may work in  the more general set-up of a proper action of $\Gamma$ on a complex manifold
$\wt{X}$ with cocompact quotient or even restrict our attention to cocompactly supported
equivariant coherent $\Dx$ or $\Ox$-modules. 
\end{rem}

\section{Farber's abelian category and its localisation}

Up to this point, we were working with $L^p$-cohomology. Now, it is time to admit
that unless $p=2$ the objects we constructed are out of control.  
 
We will change our notations and define $\Ng=W_{l,2}(\Gamma)$ and survey some relevant
homological algebraic aspects of modules over this operator algebra.

\subsection{Hilbert $\Gamma$-modules}
Let us first briefly review a very nice construction due to
Farber and L\"uck \cite{Farb1}.  For a longer review adapted to our purposes, see \cite[pp.539-544]{Eys}. For a complete review including applications in topology and algebra, see the bible of the subject \cite{Luc}.

\begin{defi} A Hilbert $\Gamma$-module (resp. of finite type, resp. separable)
is a topological $\mathbb C$-vector space with a continuous
$\Gamma$-action which can be realized as a closed
$\Gamma$-invariant subspace of $l_2\Gamma \hat{ \otimes} H$
where $H$ is a Hilbert space (resp. of finite dimension, resp. separable).
\end{defi}

\begin{lem}
 The action of $\mathbb{C}\Gamma$ on a Hilbert $\Gamma$-module $E$ extends uniquely to an action of the $C^*$-algebra $\Ng$ 
 in such a way that the image of $\Ng$ is strongly closed in $\mathfrak{B}(E)$. 
 \end{lem}

\begin{prop}
The following categories $E_f(\Gamma)
\subset E (\Gamma)$ :
\begin{itemize}
\item Objects of $E(\Gamma)$ are triples
$(E_1,E_2,e)$ where $E_1$ et $E_2$ are
 Hilbert $\Gamma$-module and $e$
 continuous $\Gamma$-equivariant linear map. 
\item $Hom_{E(\Gamma)} ((E_1,F_2,e),(F_1,F_2,f))$
is the set of pairs
$(\phi_1:E_1\to F_1, \phi_2:E_2 \to F_2)$
of  continuous $\Gamma$-equivariant linear maps such that
$\phi_2 e = f\phi_1$ under the equivalence relation
$(\phi_1,\phi_2) \sim ( \phi'_1,\phi'_2) \Leftrightarrow
\exists T\in L_{\Gamma} (E_2,F_1), \ \phi_2'-\phi_2 =
fT$.
\item $E_f(\Gamma)$
is the full subcategory
of $E(\Gamma)$ whose objects $(E_1,E_2,e)$ have the property that
$E_2$ is of finite type ($E_1$ is then also of finite type).
\end{itemize}
are abelian categories of projective
 dimension one. The forgetful functor $\Phi$ from $E(\Gamma)$
to the category of $\Ng$-modules defined by
$\Phi((E_1,E_2,e)):= E_2/e(E_1)$ is faithful, respects direct sums, kernels and cokernels and is conservative. 
\end{prop}
 \begin{prf} See \cite{Eys}. The main point is that the proof in \cite{Farb1} does not require finite type. 
\end{prf}
\begin{rem}
It is not clear to the author whether the forgetful functor $\Phi$ is fully faithful on $E_f(\Gamma)$. 
It is fully faithful  on the full subcategory 
of 
projective modules thanks to \cite{Griffin}. Fully faithfulness would follow if  $\Phi(E)$ was a projective $\Ng$-module whenever $E$ is a finite type Hilbert $\Gamma$-module but it doesn't seem to be true. 

 \end{rem}

The following corollary greatly simplifies our treatment: 

\begin{coro}\label{algiso}
 If $f^{\bullet}: K^{\bullet} \to L^{\bullet}$ is a 
 continuous morphism of complexes of Hilbert $\Gamma$-modules whose terms are in $E(\Gamma)$, hence a morphism of complexes in $E(\Gamma)$, $f^{\bullet}$ induces a  isomorphism in $D(E(\Gamma))$ if and only if $\Phi(f^{\bullet})$ induces 
 an algebraic isomorphism  in cohomology.
\end{coro}

An object $X=(E_1,E_2,e)$ of $E_f(\Gamma)$ has two basic invariants. Its
Von Neumann dimension $\dim_{\Gamma} X\in \mathbb{R}$ depends
only on $P(X)=E_2\slash\overline{eE_1}$ and has properties similar to the
dimension function of ordinary linear  algebra and its Novikov-Shubin
invariant $NoSh(X)=(E_1,\overline{eE_1},e)$.

 \begin{rem}(Tapia) \label{tapia} This construction of an abelian category is a special case of
 \cite[pp. 20, 40-41]{BBD}. Actually Hilbert $\Gamma$-modules form an exact category,
even a quasi abelian one  as follows from \cite[section 3.2]{Sch},  which satisfies the conditions in \cite{BBD}.
The same holds with $\Ng$-Fr\'echet modules. 
\end{rem}

\subsection{$\Gamma$-Fredholm Complexes}
The main nice property  complexes of   Hilbert $\Gamma$-modules can 
have in general is being $\Gamma$-Fredholm.

\begin{defi}
 A bounded  complex of Hilbert $\Gamma$-modules (with a positive inner product) $(C^k, d_k)$  
 is $\Gamma$-Fredholm if and only if the
 spectral family $E^{dd^*+d^*d}_{\lambda}$ satisfies $\exists \lambda>0$ such that the image of $E_{\lambda}$ 
 has finite $\Gamma$-dimension.
\end{defi}

This notion  depends on the notion of a Fredholm operator given in \cite[Definition 1.20, p. 26]{Luc}. 
 It is invariant by quasi-isomorphisms in $E(\Gamma)$ thanks to \cite[Theorem 2.19 p. 83]{Luc} .
 There is a stronger notion. 

\begin{defi}
 A bounded  complex of Hilbert $\Gamma$-modules (with a positive inner product) $(\bar{C}^k, \bar{d}_k)$ is strongly $\Gamma$-Fredholm if and only if
 it is quasi-isomorphic as a complex in $E(\Gamma)$ to another complex $(C^k, d_k)$ whose
 spectral family $E^{dd^*+d^*d}_{\lambda}$ satisfies $\exists \lambda>0$ such that the image of $E_{\lambda}$ is a 
 finitely generated Hilbert $\Gamma$-module.
\end{defi}

This is a stronger notion since a finite $\Gamma$-dimensional Hilbert module need not be finitely generated (e.g. for $\Gamma=\Z$). 

\begin{ques}
 Can one drop the quasi-isomorphism? Perhaps the proof of \cite[Theorem 2.19 p. 83]{Luc} can be modified using 
 the center-valued trace. 
\end{ques}


\begin{lem}
The  homotopy category of bounded strongly $\Gamma$-Fredholm  is equivalent to $D^b(E_f(\Gamma))$. 
\end{lem}
\begin{prf} Since  the full abelian subcategory $E_f(\Gamma)\subset E(\Gamma)$ has enough $E(\Gamma)$ projective and both have finite projective dimension
$\psi:D^b(E_f(\Gamma))\to D^b_{E_f(\Gamma)}(E(\Gamma))$ is an equivalence. 

Certainly a strongly Fredholm complex has its cohomology in $E_f(\Gamma)$. 

Quasi-isomorphisms in the homotopy category of complexes in $E(\Gamma)$ are exactly the homotopy classes of morphisms of complexes that are
algebraic quasi isomorphisms
thanks to the exactness of the faithful forgetful functor $E(\Gamma)\to {Mod}_{\Ng}$. 

Since strongly Fredholm complexes are complex of projective objects in $E(\Gamma)$, the functor $\psi'$ from  the homotopy category of bounded above strongly $\Gamma$-Fredholm 
complexes to the derived category $D^{b} E(\Gamma)$ is fully faithful and takes its values in $D^b_{E_f(\Gamma)}(E(\Gamma))$. Since $\psi$ is an equivalence, whose image
is contained in the image of $\psi'$, $\psi'$ is essentially surjective.

\end{prf}

\subsection{An equivalence of categories}
There is a more algebraic approach to $E_f(\Gamma)$
 \cite[p. 288]{Luc}.  $\Ng$ is a semihereditary \cite[Theorem 6.7 p. 239]{Luc} hence coherent ring. It turns out that $E_f(\Gamma) $ is equivalent to the abelian
 category of finitely presentable  $\Ng$-modules. But the  equivalence in question, denote it by $\nu$,  is not given by $\Phi$. Indeed it is constructed  using the 
 equivalence given by the functor on finite rank free $\Ng$-modules defined  by $M\mapsto l^2\Gamma\otimes_{\Ng} M$. It is not obvious that it is an equivalence. 
 
  There is also a dimension theory for arbitrary $\Ng$-modules which generalizes $\dim_{\Gamma}$ and more or less reduces the theory of $L_2$-Betti numbers and Novikov-Shubin invariants to algebra. However 
 non zero $\Ng$-modules of dimension $0$ may exist in sharp contreast with projective Hilbert $\Gamma$-modules.

\subsection{Affiliated Operators}
The algebra of affiliated operators $\Ug$
 is a flat  extension  $\Ng\subset \Ug$ \cite[Theorem 8.2.2]{Luc} such that $\Ug \otimes  NoSh(X)=0$ whenever $X$ is an object in $E_f(\Gamma)$. It is a coherent ring, even a Von Neumann regular one. 
 So that finitely presented $\Ug$-modules form an abelian category of projective dimension $0$ (all objects are projective!). 
Furthermore $\dim_{\Gamma}$ extends to $\Ug$-modules (no topological structure needed) in such a way that  objects  of the form $\Ug\otimes_{\Ng} E$ where $E$ is a Hilbert $\Gamma$-module of finite type are finite dimensionnal, $\dim_{\Gamma}$
being preserved. In particular, for a complex $K^{\bullet}$ in $E_f(\Gamma)$, 
$$H^q(\Ug \otimes K^{\bullet})=\Ug \otimes_{\Ng} \bar H^q (K^{\bullet})=\Ug \otimes_{\Ng} P(H^q (K^{\bullet})).$$
The author does not know how to use in Complex Analytic Geometry the torsion information lost in that process. So forgetting about it seems to be 
appropriate from a pragmatic point of view. 

The  relevance of this localization process for Hodge theory is one of the main ideas of \cite{Din}. 
 There,  intermediate abelian categories fitting in a succession of exact functors of abelian categories $$E_f(\Gamma) \buildrel{faith.}\over\longrightarrow {{Mod}}(\Ng) 
 \to  {{Mod}}(\Ng)/\tau \to  {{Mod}}(\Ug)$$  are introduced where $\tau$ is a torsion theory (or an appropriate Serre subcategory).
 Here we will only consider the case $\tau=\tau_{\Ug}$ 
 in the notations of loc.cit.: when it is possible we will work in $E_f(\Gamma)$ and when it  becomes necessary we will apply the functor $\Ug\otimes_{\Ng}$. 
 
However, the reason \cite{Luc} introduces affiliated operators is his $[0,+\infty]$-valued dimension theory for $\Ug$-modules which enables him to make the theory of $L^2$-Betti numbers 
more or less algebraic. This enables one to look at non-locally finite simplicial complexes like $K(\Gamma, 1)$ 
simplifying Cheeger-Gromov's article \cite{CG}. One has however to do a minimal amount of functional analysis to prove the $\Ug$-modules we encounter are 
finite dimensional or finitely generated projective. This is why we will not just apply the functor $\Ug\otimes_{\Ng}$ to the construction of the two preceding sections with $p=2$ although it is extremely tempting.

We conclude with the following lemmas:

\begin{lem}
 $\Ug \otimes_{\Ng} \nu$ is naturally equivalent to $\Ug\otimes_{\Ng} $ on $E_f(\Gamma)$.
\end{lem}
\begin{prf}
 This follows from the construction of $\nu$ and of the relation $\Ug \otimes_{\Ng} l^2 \Gamma =\Ug$ which in turn follows from the realization of $ \Ug$ as an Ore localization of $\Ng$ \cite{Luc}.
\end{prf}

 \begin{lem} \label{ugtensclosure} Let $E$ be an object of $E_f(\Gamma)$ endowed with 3 filtrations $W$, $F$, $G$. $\bar F$ and $\bar G$ are $n$-opposed on $Gr_{\bar W ^n} P(E)$ if and only if 
 $\Ug \otimes_{\Ng} F$ and $\Ug \otimes_{\Ng} G$ are $n$-opposed on $Gr_{\Ug \otimes_{Ng} W ^n} \Ug \otimes_{\Ng}E$ 
 \end{lem}
\begin{prf}
Left to the reader. 
\end{prf}

 \subsection{A natural question}\label{selfinj}
The construction of a completely satisfying $L_2$ Mixed Hodge Theory might be eased  by the use of further  results from the theory of operators algebras.  
 A saliant feature of $\Ug$ is that it 
is self injective which is exactly, according to a remark in \cite{Luc}, what is needed for neat duality statements.
 It would be helpful if the following question had a positive answer:

\begin{ques}\label{fredcomplex}  Assume we have a complex of separable Hilbert $\Gamma$-modules (or Fr\'echet $\Ng$-modules) whose cohomology is 
isomorphic as a  $\Ng$-module to the $\Ng$-module underlying an object of $E_f(\Gamma)$. 
Is the complex strongly $\Gamma$-Fredholm? 
 \end{ques}
\subsection{Real Structures}

The $*$-algebras $\Ng$ and $\Ug$ carry a real structure, namely a conjugate 
linear algebra involutive automorphism $\dagger$ commuting
with the conjugate linear algebra involutive anti-automorphism $*$ ,   and a real structure on a module 
over these algebras is just a conjugate linear automorphism on the underlying vector space compatible with $\dagger$. For instance $l_2 \Gamma$ has a real structure. We will denote
by $\R E_f(\Gamma)$ the category of formal quotients of real Hilbert $\Gamma$-modules. 

\section{\label{sec3} Refined $L^2$-cohomology}

\subsection{Finiteness theorem for $L^2$ constructible cohomology}

\begin{prop} \label{liftcons1}
\label{co}
Let $\mathbb{S}$ be a cocompact $\Gamma$-simplicial complex.
There is a $\partial$-functor $\mathbb{H}^*_2(|\mathbb{S}|,-): D^b(\mathbf{Cons}_{K,\Gamma} (\mathbb{S}))\to
D^b (E_f(\Gamma))$ such that, upon composition with the forgetful
functor to the derived category of $\Ng$-modules
$\Phi( \mathbb{H}^*_2(|\mathbb{S}|,-))=
\mathbb{H}^*_{(2)}(|\mathbb{S}|,-)$ (cf. Definition \ref{lpco}).
\end{prop}
\begin{prf} First assume that $\mathbb{S}$ satisfies the following
technical assumption: for every pair of distinct adjacent vertices $p,q\in S$,
$ \psi(p)\not=\psi(q)$ where $\psi: S\to \Gamma \backslash S$ is the quotient map. In particular, the set of vertices  of a given simplex maps injectively
to $\Gamma \backslash S$. Choose a well-ordering of $\Gamma\backslash S$.
This provides each simplex $\sigma$ with an order $<_{\sigma}$ on its
vertices such that $<_{\gamma.\sigma}=\gamma(<_{\sigma})$ and if $\tau\subset \sigma$, $<_{\sigma}|_{\tau}=<_{\tau}$. This  defines a sign
$\epsilon_{\tau, \sigma}$ for every pair of simplices $\tau\subset \sigma$
such that ${\text{Card}}(\sigma)={\text{Card}}(\tau)+1$, namely $\epsilon_{\tau,\sigma}=(-1)^{\nu}$ where $(\sigma,<_{\sigma})=
p_0<p_1<\ldots<p_{{\text{Card}}(\sigma)-1}$ and $\sigma - \tau=\{p_{\nu}\}$.

Let $\mathcal{F}$ be an object of $\mathbf{Cons}_{\Gamma}(\mathbb S)$.
For every simplex $\sigma$, set
 $U_{\sigma}:=\cup_{\sigma\subset\tau}|\tau|$
and  $F_{\sigma}:=H^0(U_{\sigma},\mathcal{F})$.
For $\tau\subset \sigma$, the sheaf structure gives a map $\rho_{\tau,\sigma}:F_{\tau}\to F_{\sigma}$.
 Set $C_c^p(\mathbb{S}, \mathcal F):=\oplus_{{\text{Card}}({\sigma})=p+1} F_{\sigma}$ and
 for $f_{\tau}\in {\mathcal{F}}_{\tau}$,
 $$df_{\tau}= \sum_{\tau\subset \sigma, {\text{Card}}(\sigma)={\text{Card}}(\tau)-1}
 \epsilon_{\tau,\sigma} \rho_{\tau,\sigma}(f_{\tau}).$$
  This defines a complex
 of $\Gamma$-modules $C_c^{\bullet}(\mathbb{S},\mathcal{F}).$

 This complex is actually the \v Cech complex of $\pi_!{\mathcal{F}}$
 in the covering $(U'_q)_{q\in\Gamma\backslash S}$ of $\Gamma\backslash|\mathbb{S}|$ where we define $U'_q=\pi(U(p))$
 where $p\in S$ satisfies $\pi(p)=q$. Using Proposition 8.1.4 p.323 in
 \cite{KS}, we see that $l_p\Gamma \otimes_{K\Gamma} C_c^{\bullet}(\mathbb{S},\mathcal{F})$ computes $H^{\bullet}_{(p)}(|\mathbb{S}|,\mathcal{F})$. In case $p=2$; this complex
 is in fact a complex in $E_f(\Gamma)$.

 This construction is obviously functorial, and taking the simple complex
 associated to a double complex one would construct the sought-for $\partial$-functor. The technical assumption on $\mathbb{S}$ is not always
 satisfied, but it holds for the barycentric subdvision $\beta \mathbb{S}$.
 We certainly have a fully faithful forgetful functor $\mathbf{Cons}_{K,\Gamma}
 (\mathbb{S})\to \mathbf{Cons}_{K,\Gamma}
 (\beta\mathbb{S})$ and we  define a functor between categories of complexes
 to be $s(l_2\Gamma \otimes_{K\Gamma} C^{\bullet}(\beta\mathbb{S},-))$. Passing to derived categories, it  descends to
 $\mathbb{H}^*_2(|\mathbb{S}|,-)$.
\end{prf}

\begin{prop}\label{subalrefcoh}

Let $\wt{X}$ be a cocompact subanalytic $\Gamma$-space.

There is a $\partial$-functor $\mathbb{H}^*_2(\wt{X},-): D^b_{\mathbb{R}-c, \Gamma} (\wt{X}))\to
D^b (E_f(\Gamma))$ such that one has
$\Phi( \mathbb{H}^*_2(X,-))=
\mathbb{H}^*_{(2)}(X,-)$.

This functor enjoys the following properties:

\begin{itemize}
\item (Leray spectral sequence) For every proper
$\Gamma$-equivariant morphism  $f: \wt{X}\to  \wt{Y}$
$\mathbb{H}_{2}( \wt{X}, -)$ and $\mathbb{H}_{2}( \wt{Y}, -)\circ Rf_*$
are naturally isomorphic functors.
 \item (Atiyah's $L_2$ index theorem) If $\Gamma$ is fixed point free on $\wt{X}$
 $$ \sum_{i} (-1)^i \dim_{\Gamma} \mathbb{H}_2^i( \wt{X},\mathcal
{F}^{\bullet})=
 \sum_{i} (-1)^i \dim_{\mathbb C} \mathbb{H}^i(\Gamma\backslash\wt{X},\Gamma\backslash\mathcal
{F}^{\bullet}).
 $$
\end{itemize}
\end{prop}
\begin{prf} 
As in the proof of Theorem \ref{lpcocons}, the first part derives  from Proposition \ref{co}. 
The proof of the additionnal statements is  simpler than the proof of 
similar statements for coherent cohomology in \cite{Eys} and will not be given in detail.
\end{prf}

If $\wt{X}$ is complex analytic,  
we will restrict $\mathbb{H}^*_2(\wt{X},-)$ to $ D^b(\mathbf{Cons}_{\Gamma} (\wt{X}))$.
We will also denote by $\overline{\mathbb{H}}^k_2(\wt{X},-)$ the $k$-th reduced cohomology functor 
$$\overline{\mathbb{H}}^k_2(\wt{X},-) = P( H^k(\mathbb{H}^*_2(\wt{X},-)).
$$

It is a projective Hilbert $\Gamma$-module and one has $\dim_{\Gamma} \mathbb{H}_2^i( \wt{X},\mathcal
{F}^{\bullet})= \dim_{\Gamma} \overline{\mathbb{H}}_2^i( \wt{X},\mathcal
{F}^{\bullet})$. 

\subsection{\label{vncohd}Finiteness theorem for $L^2$ coherent De Rham cohomology}

Assume in this subsection that $X$ be  a compact complex manifold.
The $L^2$ coherent cohomology functor $$H^{\bullet}_{L^2}(\wt{X}, -): D^b \coh \to D^b (E_f(\Gamma))$$
defined in \cite[Th\'eor\`eme 5.3.8]{Eys} comes from a $\C$-linear functor denoted
by ${\mathcal{C}}$ from a category $A$ of  coherent $\Ox$-modules endowed with  local locally free presentations
taking values in the so-called qhtf  complexes of an additive category  $MontMod$ of so-called Montelian modules which contain $E(\Gamma)$ as a full subcategory. 

The homotopy category of qhtf complexes of Montelian modules localized with respect to a specific class of quasi-isomorphisms
is naturally equivalent to $D^b (E_f(\Gamma))$. 

The resulting functor $C^b(A) \to D^b (E_f(\Gamma))$ factors through $D^b(\coh)$.  

$\mathcal{C}$  is essentially the \v{C}ech cohomology of $l^2\pi_*$ with respect to some
Stein covering $\mathfrak{U}_0$ of $X$. 
Actually, we  can choose an appropriate 
 germ at $t=0$ of an increasing family of coverings $\mathfrak{U}_t$ defined for $t\ge 0$ 
such that $\mathfrak{U}=\mathfrak{U}_0$ and: $$C_*(t)=\mathcal{C}^{\bullet}(\mathfrak{U}_t, l^2\pi_*\mathcal{F}) \to \mathcal{C}^{\bullet}(\mathfrak{U}_{t'}, l^2\pi_*\mathcal{F})  $$ is a quasi-isomorphism for $t'\le t$. 
Then $\mathcal{C}(\mathcal{F})= (C_*(t))_{t\ge0}$  up to some inessential auxiliary data. We relegate to an appendix the more detailed discussion of these auxiliary data.

\begin{lem}
The functor $\mathcal{C}'$ extends to $C^b\mathbf{Coh}(\Ox, \diff)$
the full subcategory of $C^b\mathbf{M}(\Ox, \diff)$ whose objects are differential complexes of coherent analytic sheaves  as an additive functor.
\end{lem}

\begin{prf}
Since differential operators between coherent analytic sheaves act continuously on
the Fr\'echet space of their  sections, the arguments of \cite{Eys} apply. 
\end{prf}

Recall from \cite{Sai} that the correct notion of quasi isomorphism in the
triangulated category $K^b\mathbf{M}(\Ox, \diff)$ of homotopy classes of bounded 
complexes in $\mathbf{M}(\Ox, \diff)$ are the differential quasi isomorphisms. 
We denote by $dqi$ this localizing class which is a priori smaller than the class $qi$ of 
sheaf-theoretic quasi-isomorphisms. The class $dqi$ is needed  to invert
the De Rham functor and one has a $\partial$-functor  
$$\nu:K^b\mathbf{Coh}(\Ox, \diff)_{dqi} \to K^b\mathbf{Coh}(\Ox, \diff)_{qi}. 
$$
Here $K^b\mathbf{Coh}(\Ox, \diff)_{dqi}$ stands for the  essential image of  $C^b\mathbf{Coh}(\Ox, \diff)$ (the essential image is a stricly full category). 

\begin{coro}
 The functor $H^{\bullet}_{L^2}(\wt{X}, -)$ extends to a $\partial$-functor
$$H^{\bullet}_{2}(\wt{X}, -):K^b\mathbf{Coh}(\Ox, \diff)_{dqi} \to D^b (E_f(\Gamma))$$ such that
$\Phi\circ H^{\bullet}_{DR,2}(\wt{X}, -)$ is naturally isomorphic to
$H^{\bullet}(X, l^2 \pi_* -)$. 
\end{coro}
\begin{prf}
This follows from the fact that a complex in $E_f(\Gamma)$ is acyclic iff it is algebraically acyclic by Corollary \ref{algiso}. 
\end{prf}

We will from now on make a technical assumption, namely that the
coherent $\Dx$-modules we consider admit a  global good filtration.  A second possibility would be to work on the category $ \filtdmod$.

\begin{lem}
 Let $\mathcal{K}^{\bullet}\in Ob(K^b  \cohd)$. Let $F_.$ be a filtration of $K^{\bullet}$
inducing a good filtration on each term. Then $F_p DR(\mathcal{K}^{\bullet}) \to DR(\mathcal{K}^{\bullet})$
is a quasi isomorphism for $p\gg 0$ and $F_p DR(\mathcal{K}^{\bullet})$  is independant of $p\gg 0$ up to a unique differential quasi-isomorphism
hence defines unambiguously an object $DR'(\mathcal{K}^{\bullet})$ of $K^b\mathbf{Coh}(\Ox, \diff)_{dqi}$. This assignement is functorial.

\end{lem}

\begin{prf}
 See \cite{Sai}, in particular for the functorial behaviour of this construction. 
\end{prf}

It is tempting to believe that one could work with local good filtrations (actually with the local presentations inducing them) using simplicial gluing techniques as in \cite[Section 6]{Eys}
but we shall refrain from doing so. 

We  can thus define a $\partial$-functor 
$$ tDR:
D^b \cohd_{good filt} \to K^b\mathbf{Coh}(\Ox, \diff)_{dqi}
$$
which is compatible to the restriction to $D^b\cohd$ of Saito's equivalence:
$$DR: D^b \dmod \to K^b \mathbf{M}(\Ox, \diff)_{dqi}
$$

\begin{prop}
The functor $$H^{\bullet}_{DR, 2}(\wt{X}, -)= H^{\bullet}_{2}(\wt{X}, -) \circ tDR:
D^b \cohd_{good filt} \to D^b(E_f(\Gamma))$$
is a $\partial$-functor such that $\Phi\circ H^{\bullet}_{DR, 2}(\wt{X}, -)$ 
is naturally equivalent to the restriction of the functor
$H^{\bullet}_{DR,L^2}(\wt{X}, -)$ .
\end{prop}

Once again, one can define the reduced $L_2$ cohomology of $\mathcal{M}$ a holonomic $\Dx$-module admitting a good filtration (or a complex of such):
$$ \bar H^{k}_{DR, 2}(\wt{X}, \mathcal{M} )=P(H^{k}_{DR, 2}(\wt{X}, \mathcal{M})).
$$

\begin{rem}
 When $\mathcal{F}$ is a coherent $\Ox$-module, one can form the induced $\Dx$-module $ind(\mathcal{F}):=\Dx\otimes \mathcal{F}$, a coherent $\Dx$-module with a global good filtration,  and there is a morphism of complexes of sheaves
 $DR(ind(\mathcal{F})) \to \mathcal{F}\otimes \omega^n_X$ which is a quasi-isomorphism. We have a natural isomorphism in $D^b(E_f(\Gamma))$
 $$H^{\bullet}_{DR,2}(\wt{X}, \Dx\otimes_{\Ox} \mathcal{F}) \simeq H^q_2(\wt{X}, \mathcal{F}\otimes \omega^n_X)$$ where we use the notation of \cite{Eys} for coherent $L^2$-cohomology. 
\end{rem}

\subsection{$L^2$ Poincar\'e-Verdier Duality}

If $\mathbb{S}$ be a cocompact (in particular finite dimensionnal) $\Gamma$-simplicial complex $\mathbf{Cons}_{K, \Gamma}(\mathbb{S})$ can described combinatorially in terms of the poset $(\Sigma(\mathbb{S}), \le)$ where $\Sigma(\mathbb{S})$ is the set 
of simplices and $\sigma \le \tau$ if and only if $\sigma$ is a face of $\overline{\tau}$ . The partial order is $\Gamma$-equivariant. Then  $\mathbf{Cons}_{K, \Gamma}(\mathbb{S})$  is the category
of $\Gamma$-equivariant functors $$(\Sigma({\mathbb{S}}), \le) \to \mathrm{Finite-dimensionnal\ } K - \mathrm{Vector \ Spaces}.$$ This is nothing but a reformulation of a part of the construction in the proof of Proposition \ref{co}. 
We call the maps $F_{\tau}\to F_{\sigma}$ when $\tau \le \sigma$ the corestriction maps of $F\in \mathrm{Ob}(\mathbf{Cons}_{K, \Gamma}(\mathbb{S}))$. 

Poincar\'e-Verdier has an explicit combinatorial formulation \cite{Sh, PSch, Curry2} which is presented very efficiently in the note \cite{Curry1} and was apparently first observed by A. Shephard 
in his 1985 unpublished thesis under  R. MacPherson's direction. 

\begin{prop} Let $F$ be a $\Gamma$-constructible sheaf. Then its Verdier dual is represented by the complex of injective  $\Gamma$-constructible sheaves :
$$ \mathbb{D}(F)= \ldots \longrightarrow \bigoplus_{\dim(\sigma)=i} \iota_{\bar \sigma *}  \underline{K}_{\bar \sigma} \otimes F_{\sigma}^{\vee} 
\buildrel{\partial}\over\longrightarrow  \bigoplus_{\dim(\tau)=i-1} \iota_{\bar \tau *} \underline{K}_{\bar \tau} \otimes F_{\tau}^{\vee} \longrightarrow \ldots
$$
where 
\begin{itemize}
 \item 
$\iota_{\sigma}: \bar \sigma \to | \mathbb{S} |$ is the closed embedding,  
\item For every $i\in \N$,
$\mathbb{D}^i(F)= \bigoplus \iota_{\bar \sigma *} \underline{K}_{\bar \sigma} \otimes F_{\sigma}^{\vee}$ is placed in degree $-i$, 
\item The $(\tau, \sigma)$ matrix component of  $\partial$ is the tensor product of 
the transpose of the corestriction map $F_{\tau}\to F_{\sigma}$ when $\tau \le \sigma$ and the homology boundary map $\iota_{\bar{\sigma} *} 
\underline{K}_{\bar \sigma}  \to  \iota_{\bar{\tau} *}  \underline{K}_{\bar \tau}$ and $ 0$ if $\tau$ is not a facet of $\bar \sigma$ . 
\item $V\mapsto V^{\vee}$ is the usual duality functor on finite-dimensionnal $K$ - Vector  Spaces,  
\end{itemize}
If $F^{\bullet}$ is a bounded complex of $\Gamma$-constructible sheaves $\mathbb{D}(F^{\bullet})$ is represented by the totalisation of the double complex obtained by applying $\mathbb{D}$. 
\end{prop}

This proposition gives a lift of Verdier Duality to the category of complexes, before taking the derived category. 

Now if $F$ is a $\Gamma$-constructible sheaf, we may construct $F^{\vee}$ by applying the usual duality functor and get a $\Gamma$-constructible cosheaf, namely a {\em contravariant} functor 
$$(\Sigma({\mathbb{S}}), \le) \to \mathrm{Finite-dimensionnal\ } K - \mathrm{Vector \ Spaces}.$$ We may construct its homology chain complex $C_{\bullet}(\mathbb{S}, F^{\vee})$ , concentrated in positive degrees, 
and view it as a {\em cochain} complex $C_{-\bullet}(\mathbb{S}, F^{\vee})$ concentrated in negative degrees. These are complexes of projective $K \Gamma$-modules. 

\begin{lem}
 There is a functorial monomorphism of complexes of $K\Gamma$-modules:
 $$
  C_{-\bullet}(\mathbb{S}, F^{\vee}) \to C_c^{\bullet} (\mathbb{S}, \mathbb{D}(F))
 $$
 such that the quotient complex is acyclic. 
\end{lem}

\begin{proof} The complex $C_c^{\bullet} (\mathbb{S}, \mathbb{D}^i(F))$ is the same as 
$$ \bigoplus_{\dim(\sigma)=i} C^{\bullet} (\bar \sigma) \otimes F_{\sigma}^{\vee}
$$
 where  $C^{\bullet} (\bar \sigma)$ is the simplicial cochain complex of the closed simples $\bar \sigma$. 
 
 We have a natural cochain equivalence $K[0] \to C^{\bullet} (\bar \sigma)$ sending $1$ to the cochain which $1$ on all vertices of $\bar \sigma$. This enables to construct a $\Gamma$-equivariant cochain equivalence:
 $$ \bigoplus _{\dim(\sigma)=i} F_{\sigma}^{\vee} [0] \to \bigoplus_{\dim(\sigma)=i} C^{\bullet} (\bar \sigma) \otimes F_{\sigma}^{\vee}
 $$
 which commutes with the natural boundaries by construction. This proves the lemma since $C_c^{\bullet} (\mathbb{S}, \mathbb{D}(F))$ is the simple complex attached to the double complex $C_c^{\bullet} (\mathbb{S}, \mathbb{D}^{\bullet}(F))$. 
\end{proof}

\begin{lem} Let $F^{\bullet}$ be a bounded complex of $\Gamma$-constructible sheaves, then there is a functorial map of bounded complexes of projective Hilbert $\Gamma$-modules obtained by taking the simple complex attached to 
$$
l^2 \Gamma \otimes_{K\Gamma}  C_{-\bullet}(\mathbb{S}, F^{\bullet \vee}) \to l^2 \Gamma \otimes_{K\Gamma}  C_c^{\bullet} (\mathbb{S}, \mathbb{D}(F^{\bullet})).
$$
 
\end{lem}

\begin{proof}
 Obvious. 
\end{proof}

\begin{lem} Given $F^{\bullet}$ an object of $D^b(\mathbf{Cons}_{K, \Gamma} (\mathbb{S})$, there 
 is a functorial perfect duality of projective Hilbert $\Gamma$-modules $$ \overline{\mathbb{H}}_2^{\ i}(|\mathbb{S} | , F^{\bullet} \otimes_K\C) \otimes  \overline{\mathbb{H}}_2^{ \ -i}(|\mathbb{S} | , \mathbb{D} (F^{\bullet})\otimes_K\C ) \to \C $$
 which preserves the natural real structure if $K\subset \R$. 
\end{lem}

\begin{proof}
 This follows from the previous lemma and \cite[Lemma 2.17 (2), p. 82]{Luc}. 
\end{proof}

\begin{prop}\label{duality} Notations of Proposition \ref{subalrefcoh}. There 
 is a functorial perfect duality of projective Hilbert $\Gamma$-modules $$ \overline{\mathbb{H}}_2^{\ i}(\widetilde{X} , {F}^{\bullet}\otimes_K\C ) \otimes  \overline{\mathbb{H}}_2^{ \ -i}(\widetilde{X}  , 
 R\underline{Hom}^{\bullet} ({F}^{\bullet}, \omega^{\bullet}) \otimes_K\C) \to \C $$
 which preserves the natural real structure if $K\subset \R$ where $\omega^{\bullet}=\mathbb{D}(\underline{K}_{\widetilde{X}})$ is the Verdier dualizing complex.
 
\end{prop}

\begin{proof}
This is a restatement the previous lemma. 
\end{proof}

\begin{coro} $\dim_{\Gamma} \overline{\mathbb{H}}_2^{\ i}(\widetilde{X} , {F}^{\bullet}\otimes_K\C ) =\dim_{\Gamma}  \overline{\mathbb{H}}_2^{ \ -i}(\widetilde{X}  , 
 R\underline{Hom}^{\bullet} ({F}^{\bullet}, \omega^{\bullet}) \otimes_K\C)$.
 
\end{coro}

\begin{rem}
 It is extremely tempting to conjecture that, with the notations of \cite{Eys},  we have a $L_2$-Serre Duality theorem for Coherent analytic sheaves on complex spaces stating that there is perfect duality:
 $$\bar H^{q}_2 (\widetilde{X}, \mathcal{F}) \otimes \bar H^{-q}_2 (\widetilde{X}, R \underline{Hom}^{\bullet}_{\Ox}( \mathcal{F}, \Omega^{\bullet}_X)) \to \C, 
 $$
 where $\Omega^{\bullet}_X$ is the dualizing complex and that this even holds for coherent $\Dx$-modules if $X$ is smooth. Proving this conjecture is likely to be quite technical and does not seem promising for applications.
\end{rem}

\subsection{\label{liftrh}The comparison isomorphism}
We need to show that the comparison isomorphism which a priori lives is a quasi-isomorphism of  complexes of $\Ng$-module lifts to a quasi isomorphism of 
bounded complexes in $E_f(\Gamma)$. This is not completely trivial. However, with the notations in Theorem \ref{thm1}:
\begin{lem} If $\Sigma$ is a triangulation of $X$ refining a stratification $S$ of $X$ and $DR(\mathcal{M})$ has $S$-constructible cohomology, and if $P$ is a bounded complex of $\Sigma$ -constructible
sheaves of $\C$-vector spaces then one represent the quasi-isomorphism $\alpha$ by a morphism of complexes $\widetilde{\alpha}:P \to F_p DR(\mathcal{M})$ composed with the natural 
quasi-isomorphism $ F_p DR(\mathcal{M}) \to DR(\mathcal{M})$ for some $p\gg 1$. 
 
\end{lem}
\begin{proof}
 This follows from \cite[Prop 8.1.9]{KS}. 
\end{proof}

\begin{coro}
 If $\mathfrak{U}$ is a finite covering by Oka-Weil domains such that:
 \begin{itemize}
  \item 
it refines the covering $\mathfrak{V}$ of $X$ by the stars of the vertices of $\Sigma$
\item the non empty intersections are contractible,
 \end{itemize} 
 we have a quasi-isomorphism of $\Ng$-Fr\'echet modules, the leftmost two being in $E_f(\Gamma)$:  
 $$  \mathcal{C}^{\bullet}(\mathfrak{V}, P) \longleftarrow \mathcal{C}^{\bullet}(\mathfrak{U}, P) \buildrel{\wt{\alpha}}\over\longrightarrow \mathcal{C}^{\bullet}(\mathfrak{U}, F_p DR(\mathcal{M})).
 $$
\end{coro}

Now we have a model of $\mathbb{H}_{2,DR}(\wt{X}, \mathcal{M}$ which is a bounded complex of finite type projective Hilbert $\Gamma$-Modules with a quasi isomorphism:
$$ L^{\bullet} \to \mathcal{C}^{\bullet}(\mathfrak{U}, F_p DR(\mathcal{M})).
$$

Since the left hand side underlies a qhtf complex of Montelian modules, by \cite[Proposition 4.4.14]{Eys},  one constructs a morphism of complexes of projective Hilbert $\Gamma$-module which is a 
quasi isomorphism $$
\mathcal{C}^{\bullet}(\mathfrak{V}, P)\to M^{\bullet}
$$
and is the promised  lift of $rh_{2}$ to an isomorphism in $D^b(E_f(\Gamma))$. The functoriality of the construction is left to the reader. 
This concludes  the proof of Theorem \ref{thm1}. 

\subsection{$L_2$-cohomology of Mixed Hodge Modules} 
Now, $X$ is a complex projective manifold.

Let $\mathbb{M}$ be a Mixed Hodge Module in the sense of \cite{SaiMHM}. It is a triple $$((\mathcal{M}, F, W), (\mathbb{M}^B, W^B), \alpha)$$ where:
\begin{itemize}
 \item $(\mathcal{M}, F, W)$ is a bifiltered $\Dx$-module (which is regular holonomic),
 \item  $(\mathcal{M},F)$ is a good filtration,
 \item   $\mathbb{M}^B$ is a perverse sheaf over  $\Q$,
 \item $W_B$ is a filtration of $\mathbb{M}^B$ in the abelian category of perverse sheaves,
 \item $\alpha: DR(\mathcal{M})\to \mathbb{M}^B\otimes_{\Q} \C$ an isomorphism in $D^b_c(X, \C)$, actually a filtered quasi isomorphism if the  weight filtrations are 
taken into account.
\end{itemize}

all these data satisfying some non-trivial conditions.  

\begin{propdefi}
We can define in $E_f(\Gamma)$,  a real structure and a real weight filtration $W$ on the $\Ng$-module $\mathbb{H}_2^{k}(\wt{X},\pi^{-1} \mathbb{M}^B )$
 by taking the image of the functorial morphism $$ \mathbb{H}_2^{k}(\wt{X},\pi^{-1} W_{\bullet} \mathbb{M}^B ) \longrightarrow \mathbb{H}_2^{k}(\wt{X},\pi^{-1} \mathbb{M}^B )$$ 
 and  a filtration $F_{DR,2}$ on $H^k_{DR,2}(\wt{X}, \mathcal{M})$ by taking the image of the natural map
 $$ H^k_2 ( \wt{X}, F_{\bullet} DR(\mathcal{M})) \to H^k_2 ( \wt{X},  DR(\mathcal{M})). 
 $$ 
 
 Transporting the $F_{DR,2}$
 filtration by the isomorphism $rh_2$ induced by $\alpha$ between these objects of $E_f{\Gamma}$, 
 we get  a real structure
, a real  $W$-filtration, the weight filtration,  and a complex  filtration $F$, which we shall call the algebraically defined Hodge filtration,  on $H^k_2(\wt{X}, \mathbb{M}):=\mathbb{H}_2^{k}(\wt{X},\pi^{-1} \mathbb{M}^B )$. 

There is a perfect duality of the Hilbert $\Gamma$-modules $\bar H^{q}_2 (\widetilde{X}, \mathbb{M})$ and $ \bar H^{-q}_2 (\widetilde{X}, \mathbb{D}(\mathbb{M}))$.  
\end{propdefi}
\begin{proof} This is a direct application of our construction. 
 The last statement follows from Proposition \ref{duality}. 
\end{proof}

The following implies Conjecture \ref{theconj} in the introduction. 

\begin{conj}\label{L2hlpack} After taking reduced cohomology and closure of $W, F$, the real structure,
 the weight filtration and the algebraically defined Hodge filtration on $H^k_2(\wt{X}, \mathbb{M})$ are the constituents of a functorial graded polarisable Mixed Hodge structure in the abelian category $\R E_f(\Gamma)$. 
 
 The mixed Hodge numbers (namely the dimension of its $I^{p,q}$) obey the same  restrictions as in \cite{DelH2,DelH3,SaiMHM}. 
 
 If $\mathbb{M}$ is pure polarized,  $L$ is the cup product by a Hodge class, and $S$
 is a polarization  defined by the combination of a Saito polarization and $L_2$-Poincar\'e Verdier duality,
$(\bigoplus_k \overline{H^k_2}(\wt{X}, \mathbb{M}), L, S)$ is a polarized Hodge-Lefschetz in $\R E_f(\Gamma)$ in the sense of \cite[Part 0, Chapter 3]{MHMproject}.
\end{conj}

In the rest of the article we will see what can be done in that direction using only standard results.  
To  establish the Mixed Hodge structure,  it is enough to prove that, after tensoring with $\Ug$,
$F$ and $F^{\dagger}$ become $n$-opposed in $Gr_W^n$, hence that the tensor product with $\Ug$ is a $\Ug$-Mixed Hodge Structure thanks to Lemma \ref{ugtensclosure}. 
The Hodge Lefschetz structure seems to require that the construction of the Hodge filtration
is compatible with \cite{CKS}\cite{KK}. The duality statement survives after tensoring with $\Ug$ thanks to the duality anti-equivalence 
on finitely generated  $\Ug$-modules given by $M \mapsto M^{\vee}= Hom_{\Ug}(M, \Ug)$ (recall that $\Ug$ is selfinjective and that all finitely generated $\Ug$ modules are projective )  the following form:
\begin{lem}
 There is natural isomorphism 
 $$\Ug\otimes_{\Ng} \overline{H^k_2}(\wt{X}, \mathbb{M}) \to (\Ug\otimes_{\Ng} \overline{H^{-k}_2}(\wt{X},\mathbb{D}( \mathbb{M})) ^{\vee}.
 $$
\end{lem}

\begin{prf}
 This follows from  Proposition \ref{duality}. 
\end{prf}


\section{Analytical $L_2$ Hodge Structures}

\subsection{Complex polarized VHS on complete K\"ahler manifolds}
\begin{defi} Let $M$ be a complex manifold. 
A quadruple $(M,{\mathbb V},
F^.,S)$ is called a complex polarized variation of Hodge 
structure (a  VHS) iff ${\mathbb 
V}$ 
is a 
flat bundle of finite dimensionnal complex 
vector spaces with flat 
connection  $D$,
 $F^.$
a deacreasing filtration by holomorphic
 subbundles of ${\mathbb V}$ indexed by 
integers 
and $S$ a flat non degenerate $(-1)^w$-hermitian  
pairing such that 
\begin{enumerate}
\item The $C^{\infty}$ vector bundle $V$ 
associated to ${\mathbb V}$
 decomposes as a direct sum
$V= \oplus_{p+q=w} H^{p,q}$  with 
$F^P = \oplus_{p\ge P} H^{p,q}$
\item $ p \not = r) \Rightarrow 
S(H^{p,q}, H^{r,s})=0$ and
$(\sqrt{-1})^{p-q} S $ is positive definite on $H^{p,q}$
\item $D^{1,0} F^p \subset
 F^{p-1} \otimes \Omega^{1,0}_M$
\end{enumerate}
\end{defi}

The subbundle $H^{p,q}$ can be given an 
holomorphic structure
by the isomorphism $H^{p,q}\to F^p/F^{p+1}$. 
Denote by $d''_{p}$ the corresponding
Dolbeault operator and set $d''=\oplus_p 
d''_{p}$.
$D^{1,0}$ induces a $C^{\infty}$-linear map 
$\nabla'_p: H^{p,q} \to 
H^{p-1, q+1} \otimes \Omega^1$ called the 
Gauss-Manin connection 
and set $\nabla'=\oplus _p \nabla'_p$.
The Hermitian metric $H=\oplus_p (\sqrt{-1})^{p -q}
S_{H^{p,q}}$ will be called the 
Hodge metric. The triple $({\mathbb V},d'',\nabla')$ is a 
Higgs bundle. 

Following Deligne, we define $E^{P,Q}(\mathbb V)= \oplus_{p+r=P,s+q=Q} 
H^{p,q}\otimes E^{r,s}$ and 
$D''=d''+\nabla'$. One also defines $E^k(\mathbb V)=\oplus_{r+s=k}\oplus_{p,q} H^{p,q}\otimes E^{r,s}$.  It follows
that
$D''E^{P,Q}(\mathbb V) \subset E^{P,Q+1}(\mathbb V)$. 
Then, see \cite{Zuc},  
given any K\"ahler metric $\omega_{\wt{X}}$
 on $M=\wt{X}$,  taking formal adjoints of 
differential operators
with respect to this K\"ahler metric and 
the Hodge metric on $\mathbb V$, 
the usual K\"ahler identities hold. 

If furthermore the metric $\omega_{\wt{X}}$ is complete then the Dirac operators $D''+\mathfrak{d}'', D+\mathfrak{d}, \ldots$ 
and the Laplace operator $\Delta_D=2\Delta_{D''}=2\Delta_{D'}$ are formally self-adjoint unbounded operators on the Hilbert 
space of $L_2$ forms with values in $ \mathbb V$ see  \cite[Section 5.1]{Eys1} in this case or \cite{BL} and the references therein for the general theory. Thanks to  \cite[Chap. VIII, Theorem 3.2]{Dem}, it also follows that
the closure of $D, D'', D'$ is given by the na\"ive ansatz (namely the domain   of $D$ is the space of globally $L_2$ forms $\phi$
such that $D \phi$ taken  in the sense of distributions is globally $L_2$),  the Hilbert
space adjoints of $D, D'', D'$ are given by the na\"ive adjoints (namely the domain   of $\mathfrak{d}$ is the space of globally $L_2$ forms $\phi$
such that $\mathfrak{d} \phi$ taken  in the sense of distributions is globally $L_2$) and that the $L_2$ decomposition theorem 
 holds replacing images of $D, D', D''$ and their adjoints by their closure , namely we have an orthogonal decomposition:
 
 $$L^2(\wt{X},E^k(\mathbb V))=\mathcal{H}^k(\wt{X}, \mathbb{V})\oplus \overline{\mathrm{Im}(D)}\oplus \overline{\mathrm{Im}(\mathfrak{d})},
 $$
where $\mathcal{H}:=\ker(\Delta_D)$ is the space of $L_2$ harmonic forms and similarly for $D''$. 
 
   The $L_2$ De Rham complex $L^2DR^{\bullet}(\wt{X}, \mathbb{V})$ (resp. its Dolbeault counterpart) is the complex of bounded linear operators obtained by restriccting $D$ (resp. $D''$) to its domain.
The $L_2$ de Rham cohomology groups
 $\ker(D)/D Dom(D)$ (resp. their $L_2$-Dolbeault counterparts) are not represented by harmonic forms but the reduced cohomology groups
  $\ker(D)/\overline{D Dom(D)}$  (resp.) are. 
  
\begin{lem}
The $k$-th reduced $L^2$ 
cohomology of the complete K\"ahler manifold $\wt{X}$ with coefficients in the VHS
${\mathbb V}$ has a Hodge 
structure of weight $w+k$.
\end{lem}
\begin{proof}
 It follows from the fact $\Delta_{D}=2\Delta_{D''}$ commutes with the decomposition in $(P,Q)$ type. 
\end{proof}

\begin{lem}
 The Hodge-Lefschetz package holds for the 
 reduced $L^2$  
cohomology of the complete K\"ahler manifold $\wt{X}$ with coefficients in the VHS
${\mathbb V}$. More precisely $(\overline{H}_2^{*} (\wt{X}, \mathbb{V}), L)$ is a Hodge-Lefschetz structure polarized by $\int_{\wt{X}}S( - \wedge -))$ in the sense of \cite[Part 0, Chapter 3]{MHMproject}.
\end{lem}

\begin{proof}
 The usual proof applies. 
\end{proof}

Since $\Delta_D$ is essentially self-adjoint there exists a spectral decomposition $$\Delta_D=\int_{0}^{\infty} \lambda dE_{\lambda}$$
where  $(E_{\lambda})_{\lambda>0}$ is the spectral family of $\Delta_D$, an increasing orthonormal projector-valued function on $[0,+\infty[$ converging strongly to $\mathrm{Id}$ . The support of this spectral 
projector valued measure $dE_{\lambda}$ is the spectrum of $\Delta_D$. $E_0$ is the Hilbert space
projector on the closed subspace  $\mathcal{H}:=\ker(\Delta_D)$ and $E_{\lambda}$ is the projector on the space 
of $L_2$ forms $\phi$ such that: $$\forall n\in\N \quad <\Delta^n_D \phi,\phi > \le \lambda^n <\phi, \phi>.$$  The $E_{\lambda}$ commute with 
decomposition in $(P,Q)$-type and actually with all the operators $D,D',D'', \mathfrak{d}, \ldots, L, \Lambda$. 
The statement that $E_{\lambda}$ commutes with
a differential operator means in particular that it preserves its domain. 

For future use, we record the following more precise notation, for every $\lambda >0$:

$$ E^k_{\lambda}(\wt{X}, \mathbb{V})= \mathrm{Im} (E_{\lambda})  \cap L^2DR^{k}(\wt{X}, \mathbb{V}).
$$

This gives a  subcomplex of the $L^2$ De Rham complex:

$$ E^{\bullet}_{\lambda}(\wt{X}, \mathbb{V})= ( \ldots E^k_{\lambda}(\wt{X}, \mathbb{V}) \buildrel{D}\over\longrightarrow E^{k+1}_{\lambda}(\wt{X}, \mathbb{V}) \to \ldots ). 
$$

This first order differential operators have closed range if and only if $E_{\epsilon}=E_{0}$ for some $\epsilon>0$
namely if and and only $0$ is isolated in the spectrum of the Laplace operator. This fails for instance on the complex  line.

The natural analog of the space of smooth forms  in the compact case is the following subcomplex of $L^2DR^{\bullet}(\wt{X}, \mathbb{V})$:
$$
L^2DR_{\infty}^{\bullet}(\wt{X}, \mathbb{V})=(\bigoplus_k \bigcap_{n>0} Dom(\Delta_D^n|_{L^2(\wt{X},E^k)}), D). 
$$

It is a complex of $\Ng$-Fr\'echet spaces and we have $L^2DR_{\infty}^{\bullet}(\wt{X}, \mathbb{V})\subset C^{\infty, \bullet}(\wt{X}, \mathbb{V})$ 
by standard elliptic estimates.   The same construction works also for the Dolbeault complex. See \cite{BL} for a wider perspective. 

\begin{lem}\label{qiso}
Assume $\lambda'>\lambda>0$. Then the following inclusions of complexes:

$$
E^{\bullet}_{\lambda}(\wt{X}, \mathbb{V}) \subset E^{\bullet}_{\lambda'}(\wt{X}, \mathbb{V})\subset L^2DR_{\infty}^{\bullet}(\wt{X}, \mathbb{V}) \subset 
L^2DR^{\bullet}(\wt{X}, \mathbb{V}).
$$

are quasi-isomorphisms. In fact $E^{\bullet}_ {\lambda}(\wt{X}, \mathbb{V})$ is a homotopy retract of the three other complexes.

The same holds for the $L^2$-Dolbeault complex of a $\Gamma$-equivariant holomorphic hermitian vector bundle. 
\end{lem}
\begin{prf}
 Define $g=\int_{\lambda}^{\infty} \mu^{-1} d E_{\mu}$. Then $g$, a continuous linear operator,  
 preserves all the 4 complexes above and so does 
 $h=\mathfrak{d}g$. Now, one has $[D,h]=Id-E_{\lambda}$. The proof works for the Dolbeault complex too, using the Dolbeault laplacian 
 and $\mathfrak {d}''$. 
\end{prf}

Hence $E^{\bullet}_{\lambda}(\wt{X}, \mathbb{V})\to L^2DR^{\bullet}(\wt{X}, \mathbb{V})$ is an  isomorphism
in the derived category of the abelian  category of formal quotients of Hilbert spaces (aka separable Hilbert $\{1\}$-modules). 
And 
$E^{\bullet}_{\lambda}(\wt{X}, \mathbb{V})\to L^2DR_{\infty}^{\bullet}(\wt{X}, \mathbb{V})$ is an isomorphism in the derived category constructed in \cite{Sch}.

We endow the 4 complexes in Lemma \ref{qiso} with filtration induced by the Hodge filtration $F^p=\oplus_{P\ge p} E^{P,Q}(\mathbb V)$ on $L^2(\wt{X},E^{k}(\mathbb V))$.
This filtration is in each degree a closed subspace which is furthermore a summand. Actually the first three complexes 
are bigraded in the usual fashion. 

\begin{lem}\label{filtqiso}
 The first two inclusions of Lemma \ref{qiso} are filtered quasi-isomorphisms. 
\end{lem}
\begin{prf} We have to prove that the maps between  the $F$-exact sequences are isomorphic at the $E_1$ page. 
The usual proof does work perfectly well for the first three complexes.  Indeed $Gr_F E^{\bullet}_{\lambda}(\wt{X}, \mathbb{V})=(E_{2{\lambda}}(\Delta_{D''}), D'')$ 
and $Gr_F L^2DR_{\infty}^{\bullet}(\wt{X}, \mathbb{V})=L^2Dolb_{\infty}^{\bullet}(\wt{X}, \mathbb{V})$ whose cohomology are
isomorphic by the Dolbeault version of Lemma \ref{qiso}.
\end{prf}

\begin{rem} For the third one, it seems to be more delicate.  One has:
$$Gr_F L^2DR_{\infty}^{\bullet}(\wt{X}, \mathbb{V})\subsetneqq Gr_F L^2DR^{\bullet}(\wt{X}, \mathbb{V})\subsetneqq L^2Dolb^{\bullet}(\wt{X}, \mathbb{V}).$$
Using $\mathfrak{d}''g$ we get obtain a quasi-isomorphism of the two extreme complexes with $Gr_F E_{\lambda}(\wt{X}, \mathbb{V})$, hence the natural inclusion is a quasi isomorphism between them.
The problem is that $\mathfrak{d}''g$ does not seem to preserve $Gr_F^P  L^2DR^{\bullet}(\wt{X}, \mathbb{V})$. This group contains $Dom(D')\cap Dom(D'')$
which is preserved but the inclusion may be strict. 
\end{rem}

However, the classical case applies without any modification under a strong hypothesis that fails in the simplest case of the universal covering space of a genus one curve:

\begin{lem}
Zero is isolated in the spectrum of $\Delta_D$ if and only if $(E_0,0)\subset (E_{\lambda}, D)$ is a quasi-isomorphism.
\end{lem}

\begin{lem}
 If zero is isolated in the spectrum of $\Delta_D$ then:
 \begin{enumerate}
  \item The  decomposition theorem  is valid without taking the closure of $Im(D)$, $Im(\mathfrak{d})$. Namely, $Im(D)$ and $Im(\mathfrak{d})$ are $L^2$-closed and: 
   $$L^2(\wt{X},E^k(\mathbb V))=\mathcal{H}^k(\wt{X}, \mathbb{V})\oplus {\mathrm{Im}(D)}\oplus {\mathrm{Im}(\mathfrak{d})}, 
 $$
 and also we have an equivariant decomposition as a direct sum of closed Fr\'echet subspaces:
 $$ L^2DR_{\infty}^{k}(\wt{X}, \mathbb{V})=\mathcal{H}^k(\wt{X}, \mathbb{V}) \oplus D(L^2DR_{\infty}^{k-1}(\wt{X}, \mathbb{V}))\oplus  \mathfrak{d}(L^2DR_{\infty}^{k+1}(\wt{X}, \mathbb{V})).
 $$
  \item The  decomposition theorem for the $L^2$ Dolbeault complex is valid without taking the closure of $Im(D'')$, $Im(\mathfrak{d}'')$.
  \item The  decomposition theorem for the $L^2$ $D'$ complex is valid without taking the closure of $Im(D')$, $Im(\mathfrak{d}')$.
  \item The $D'D''$ lemma holds. Namely, 
  $$   \phi \in L^2DR_{\infty}^{k}(\wt{X}, \mathbb{V}) \cap Im(D')\cap Im(D'') \Rightarrow \exists \psi \in L^2DR_{\infty}^{k-2}(\wt{X}, \mathbb{V}) \   \phi=D'D'' \psi.
  $$
  \item The Hodge to De Rham spectral sequence of $L^2DR_{\infty}^{k}(\wt{X}, \mathbb{V})$ degenerates at $E_1$ and $D$ is $F$-strict. 
  \item The Hodge to De Rham spectral sequence of $E^{\bullet}_{\lambda}(\wt{X}, \mathbb{V})$ degenerates at $E_1$ and $D$ is $F$-strict. 
 \end{enumerate}

\end{lem}

\subsection{Polarized VHS on Galois covering spaces of compact K\"ahler manifolds}
Let $X$ be a compact K\"ahler manifold and $(X, \mathbb{V}, F^., S)$ be a polarized complex Variation of Hodge Structure of weight $w$.  Assume $\wt{X}$ is a Galois covering space of $X$ so that 
its Galois group 
$\Gamma$ acts properly discontinuously by automorphisms on $(\wt{X},\pi^{-1}\omega_{{X}}, \pi^{-1}\mathbb{V}, \pi^{-1} F^.,\pi^{-1} S)$. 
Then it is easy to see that all the Hilbert spaces considered in the previous section are separable projective $\Gamma$-modules and as such 
are endowed with a $\Ng$-module structure. Furthermore if the VHS is real the $E_{\lambda}$ and the 
De Rham $L_2$ cohomology groups carry a natural real structure. Basic elliptic theory gives:

\begin{prop}
The $L^2$-De Rham complex $L^2DR^{\bullet}(\wt{X},\pi^{-1}\mathbb{V})$ is strongly $\Gamma$-Fredholm.           
\end{prop}

\begin{prf}
This is essentially in \cite{Ati}. One can construct a $\Gamma$ equivariant parametrix namely a $L^2$ bounded 
 $\Gamma$-equivariant operator $$L^2DR^{\bullet}(\wt{X}, \pi^{-1}\mathbb{V})\to L^2DR^{\bullet}(\wt{X}, \pi^{-1}\mathbb{V})[-1]$$ such that $[D,P]=I-S$ where $S$ is a smoothing operator. 
 This also follows from \cite{Shu} which applies
 to any  elliptic complex (including the case of operators!). 
 For the reader's convenience, we will however 
 give an easy argument. 
 
 Let $(\phi_a)_{a\in A}$ be a finite family of smooth real functions such that $$ \sum_{a\in A} \phi_a^2=1$$ 
 and $\mathrm{Supp}(\phi_a) \subset U_a$ where $U_a$ is an open subset of $X$ small enough so that $\pi^{-1}(U_a) \cong \Gamma \times U_a$. 
 
 If $\psi \in  L^2DR^{k}(\wt{X}, \pi^{-1}\mathbb{V})$,   $\phi_a \psi$ identifies with an element with compact support in $L^2\Gamma \hat\otimes L^2DR^k(U_a, \mathbb{V})$ (Hilbert space tensor product)
 which we may prolongate by $0$ to an element $\overline{\phi_a\psi}$ of $L^2\Gamma \hat \otimes L^2DR^k(X, \mathbb{V})$. By construction the map $\Phi$:
 
 $$ L^2DR^{k}(\wt{X}, \pi^{-1}\mathbb{V}) \to (L^2 \Gamma)^A \hat \otimes L^2DR^k(X, \mathbb{V}) \quad \psi \mapsto \Phi(\psi) = (\overline{\phi_a\psi})_{a\in A}
 $$
 
 is a $\Gamma$-equivariant Hilbert space isometric (hence closed) embedding. 
 
 Let $\sigma_{D+D^*}$ be the symbol of the operator $D+D^*$. For every $\psi$ in the domain of $D+D^*$ on $L^2DR^{k}(\wt{X}, \pi^{-1}\mathbb{V})$ we have:
 
 $$(D+D^*) \phi_a \psi = \phi_a (D+D^ *) \psi +\pi^{-1}\sigma(d\phi_a) \psi. $$
 
 Summing up, we obtain 
 $$  \|( D+D^*) \psi \| +K \|\psi \| \ge \| \mathrm{Id}_{L^2 \Gamma ^A} \otimes (D+D^*) (\overline{\phi_a\psi})_{a\in A} \|  \ge  \| (D+D^*) \psi \| - K \|\psi \| 
 $$
 
 where $K= Card (A) \max_{x\in X} \| \sigma_x \|$. 
 
 Assume now $\psi \in E^k_{\lambda^2} (\wt{X}, \pi^{-1}\mathbb{V}) $. Then $\| (D+D^*) \psi  \| \le \lambda \| \psi \| $. Hence 
 $$ \| \mathrm{Id}_{pr_{\mu}^k \circ \Phi} \otimes (D+D^*) (\overline{\phi_a\psi})_{a\in A} \|  \le (\lambda + K) \|\psi\|..
 $$
 
 Introduce the tensor product by $\mathrm{Id}_{(L^2\Gamma)^A}$ of the spectral projector $E_{\mu}$ of $(v, \mathbb{V}$ : 
 $$ pr^k_{\mu}: (L^2\Gamma)^A \hat \otimes L^2DR^k (X, \mathbb{V}) \to (L^2\Gamma)^A \hat \otimes  E_{\mu}^k  (X, \mathbb{V}).
 $$
 
 Then if $\sqrt{\mu} > \lambda +K$ we have $\| pr_{\mu}^k \circ \Phi (\psi)\| \ge \epsilon \|\psi\|.$ for $\epsilon=\sqrt{\mu -\lambda-K} >0$.
 
 It follows that we have a closed embedding of Hilbert $\Gamma$-modules: 
 
 $$ pr_{\mu}^k \circ \Phi : E^k_{\lambda}(\wt{X}, \pi^{-1}\mathbb{V}) \to (L^2 \Gamma)^A \hat\otimes E_{\mu} ^k (X, \mathbb{V}).
 $$
 
 Since $E_{\mu} ^k (X, \mathbb{V})$ is a finite dimensional vector space by standard elliptic theory it follows that $E^k_{\lambda}(\wt{X}, \pi^{-1}\mathbb{V})$
 is a finitely generated Hilbert $\Gamma$-module for every $\lambda \ge 0$. We conclude using:

\begin{lem} The $L_2$ De Rham and Dolbeault complexes are $\Gamma$-Fredholm (resp. strongly) if and only if there exists $\epsilon>0$ such that $E_{\epsilon}$ is 
a finite $\Gamma$-dimensinal (resp. finite type) projective module.
 \end{lem}
\begin{prf}
 We leave this exercise to the reader. See \cite{Luc}. 
\end{prf}
 
\end{prf}

In the most  general relevant case, $\Gamma$ need not act in a cocompact fashion, hence we have to add the $\Gamma$-Fredholm hypothesis to state the following:

\begin{lem} Assume $L^2DR^{\bullet}(\wt{X},\pi^{-1}\mathbb{V})$ is $\Gamma$-Fredholm.
The following inclusions of complexes:

$$
E^{\bullet}_{\lambda}(\wt{X}, \pi^{-1}\mathbb{V}) \subset E_{\lambda'} (\wt{X}, \pi^{-1}\mathbb{V}) \subset L^2DR_{\infty}^{\bullet}(\wt{X}, \pi^{-1}\mathbb{V}) \subset L^2DR^{\bullet}(\wt{X}, \pi^{-1}\mathbb{V})
$$

are quasi-isomorphisms of complexes of Hilbert (resp. Fr\'echet for the third one) $\Gamma$-modules and define the same element of $D^b(E_{sep}(\Gamma))$ (resp. of the derived category of the exact category of $\Ng$-Fr\'echet modules)
all of whose cohomology groups have finite $\Gamma$-dimension. 
 
\end{lem}

\begin{lem} Assume $L ^2DR^{\bullet}(\wt{X},\pi^{-1}\mathbb{V})$ is $\Gamma$-Fredholm.
 
 \begin{enumerate}
  \item The  decomposition theorem  is valid. Namely,  we have a direct sum decomposition of $\Ug$ modules
   $$ \Ug \otimes_{\Ng} L^2(\wt{X},E^k(\pi^{-1}\mathbb V))=  \Ug \otimes_{\Ng}\mathcal{H}^k(\wt{X}, \pi^{-1}\mathbb{V})\oplus {\mathrm{Im}(D)}\oplus {\mathrm{Im}(\mathfrak{d})}.
 $$
 \item The  decomposition theorem for the corresponding $\Ug$-Dolbeault complexes is valid.  
  \item The  decomposition theorem for the $\Ug$-$D'$ complexes is valid.
  \item $\Ug \otimes_{\Ng} E_{0}^{\bullet} (\wt{X}, \pi^{-1}\mathbb{V}) \subset \Ug \otimes_{\Ng} E_{\lambda}^{\bullet} (\wt{X}, \pi^{-1}\mathbb{V}) $ is a filtered quasi-isomorphism where $\lambda>0$.
  \item The Hodge to De Rham spectral sequence of  $\Ug\otimes_{\Ng}  E_{\lambda}^{\bullet} (\wt{X}, \pi^{-1}\mathbb{V})  $ degenerates at $E_1$ and $D$ is $F$-strict.
  \item The Hodge to De Rham spectral sequence of  $\Ug\otimes_{\Ng}  L^2DR_{\infty}^{\bullet} (\wt{X}, \pi^{-1}\mathbb{V})  $ degenerates at $E_1$ and $D$ is $F$-strict.
  \item Dingoyan's  $D'D''$ lemma holds.  Namely, 
  $$   \phi \in L^2DR^{k}(\wt{X}, \mathbb{V}) \cap Im(D')\cap Im(D'') \Rightarrow \exists \psi \exists u \in \Ug^{\times} \   u\phi=D'D'' \psi.
  $$
  
 \end{enumerate}

\end{lem}
\begin{prf}
As in \cite{Din}, 
(1), (2), (3) follow from \cite[Lemme 2.15]{Din} and the fact these complexes are $\Gamma$-Fredholm.   
Observing that the formation of the cohomology of a complex commutes with $\Ug \otimes_{\Ng}$, since $\Ng \to \Ug$ is flat,  (4) follows from the fact that we have an isomorphism on cohomology after tensoring with $\Ug$, 
and we also have an isomophism on cohomology after passing to $Gr_F$ since the $L^2$-Dolbeault complex is $\Gamma$-Fredholm too ($\Gamma$-Fredholmness means that $E_{\lambda}$ has finite $\Gamma$-dimension for $\lambda>0$ small enough). 
(5) follows from (4) and the fact that the statement is trivially true for $\lambda=0$ and invariant by filtered quasi-isomorphism. (6) follows in the same way from Lemma \ref{filtqiso} and (5). (7) follows by an easy adaptation of the argument of \cite[Lemma 3.13]{Din}.
 
\end{prf}

\begin{theo} The $k$-th cohomology group of $ \Ug \otimes_{\Ng} L^2DR_{\infty}^{\bullet}(\wt{X}, \pi^{-1}\mathbb{V}) $  carries a $\Ug$-Hodge structure of weight $k+w$ which we call the analytic Hodge  filtration. 

It gives rise to  a weight $w+k$  real Hodge structure on $\Ug \otimes_{\Ng} H^{k}_2(\wt{X},  \mathbb{M})$.

Every K\"ahler class on $X$ induces a Hodge-Lefschetz isomorphism $$L^k: \Ug \otimes_{\Ng} H^{\dim(X)-k}_2(\wt{X},  \mathbb{V}) \to \Ug \otimes_{\Ng} H^{\dim(X)+k}_2(\wt{X},  \mathbb{V}). $$ 
\end{theo}

\subsection{$\Ug$-Hodge Complex}

Now its is time to compare with the previous construction. 

We can construct on $X$ a resolution of $l^2\pi_*\pi^{-1}\mathbb{V}$ 
by the sheafified $L^2$ De Rham complex. This is a complex of sheaves $l^2 DR^{\bullet} (\mathbb{V})$ whose value over $U\subset X$ is given in degree $k$ by  
$$l^2 DR^{\bullet} (\mathbb{V})(U)= \{ \omega \in L^2_{loc} (\pi^{-1} (U), E^k(\pi^{-1}\mathbb{V})) \ \forall K\Subset U \ \int_{K} \|\omega\|^2 +\| D\omega\|^2 <+\infty \}.  $$

One can construct a  polarizable Hodge module $\mathbb{M}=\mathbb{M}_{X}(\V)$ such that $\mathbb{M}^{B}=\mathbb{V}[\dim(X)]$, 
with a trivial  $W$-filtration, 
the underlying filtered $\Dx$-module the $\mathcal{V}=\mathbb{V}\otimes_{\C} \Ox$ endowed with the filtration $F^.$ made increasing, 
the underlying perverse sheaf is $\mathbb{V}[\dim(X)]$ and the comparison morphism $\alpha$ is the usual resolution $\mathbb{V}[\dim(X)] \to DR(\mathcal{V})$.

\begin{prop} There is a natural filtered quasi isomorphism of complexes in $E(\Gamma)$
$$ CD_2: (\mathcal{C}( DR(\mathcal{V}), F_.) \to (L^2DR^{\bullet}_{\infty}(\wt{X}, \pi^{-1} \mathbb{V}), F_.)$$ such that the composition with the comparison morphism induced by $\alpha$: 
$$rh_2: \mathbb{H}^{\bullet}_2(\wt{X}, \pi^{-1}\mathbb{V}) \to \mathcal{C}( DR(\mathcal{V}))\simeq H^{\bullet}_{DR,2}(\wt{X}, \mathcal{V} ) $$ 
is  the \v{C}ech-De Rham comparison isomorphism. 
\end{prop}
\begin{proof} First observe that $CD_2$ indeed maps into $L^2DR^{\bullet}_{\infty}$ and that there is indeed a morphism of complexes. Then, it is a routine task to check that these 
maps are quasi-isomorphisms and that they have the stated compatibility. It is a  filtered quasi-isomorphism thanks to the \v{C}ech-Dolbeault isomorphism for $Gr^F \mathcal{V}$ described in the appendix.  
 
\end{proof}

\begin{prop}
Consider a locally finite covering $\mathfrak{U}$ of $X$ by small enough Oka-Weil domains,   we can define a morphism of filtered $\Ng$-Fr\'echet complexes
$$(\mathcal{C}^{\bullet}(\mathfrak{U}, l^2\pi_* DR(\mathbb{M}^{DR}_{X}(\V))) , F)\to (L^2DR^{\bullet}_{\infty}(\wt{X}, \pi^{-1} \mathbb{V}), F_.)$$
which is a filtered quasi isomorphism of complexes of $\Ng$-modules. 
 
\end{prop}
\begin{proof}
The preceding argument gives also this. 
\end{proof} 

Using the flatness of $\Ng\subset \Ug$, we deduce:

\begin{coro} Tensoring by $\Ug$, we obtain a filtered morphism of complexes of $\Ug$-modules
 $$\Ug \otimes_{\Ng}(\mathcal{C}^{\bullet}(\mathfrak{U}, l^2\pi_* DR(\mathbb{M}^{DR}_{X}(\V))) , F)\to \Ug \otimes_{\Ng}(L^2DR^{\bullet}_{\infty}(\wt{X}, \pi^{-1} \mathbb{V}), F_.)$$
 which is a filtered quasi isomorphism. 
\end{coro}

Before summarizing the outcome of the discussion,  we need the following definition \cite{DelH3}:
\begin{defi}
 A Hodge complex of $\Ug$-modules with real structure and weight $w$ is a triple: $(A^{\bullet}, (B^{\bullet}, F), \gamma)$ where $A^{\bullet}$ is a complex of $\Ug$-modules with real structures, $\gamma: A^{\bullet} \to B^{\bullet}$ 
 an isomorphism in the derived category of $\Ug$-modules such that the $F$-spectral sequence degenerates at $E_1$ and the pair of filtrations $(F,\overline{F})$ on $H^k(B^{\bullet})$ is a Hodge structure of weight $w+k$ in the category of $\Ug$-modules. 
 
We say  $(B^{\bullet}, F)$ underlies  a  Hodge complex of $\Ug$-modules with real structure and weight $w$ if it can be completed to such a triple. 
\end{defi}

\begin{theo} \label{HodgeComplex} 
Let $X$ be a compact K\"ahler manifold and $(X, \mathbb{V}, F^., S)$ be a polarized complex Variation of Hodge Structure of weight $w$.
\begin{enumerate}
 \item The $F$-spectral sequence of $\Ug \otimes_{\Ng}(\mathcal{C}^{\bullet}(\mathfrak{U}, l^2\pi_* DR(\mathbb{M}^{DR}_{X}(\V))) , F)$ degenerates at $E_1$.
\item The algebraically defined Hodge filtration gives rise to  a weight $w+k$  real Hodge structure on $\Ug \otimes_{\Ng} H^{k}_2(\wt{X},  \mathbb{V})$ which coincides with the analytically defined one.
\item $\Ug \otimes_{\Ng}(\mathcal{C}^{\bullet}(\mathfrak{U}, l^2\pi_* DR(\mathbb{M}^{DR}_{X}(\V))) , F)$  underlies a weight $w$ Hodge complex of $\Ug$-mod\-ules with real structure and finite $\Ug$-dimensional cohomology objects.
\item Every K\"ahler class on $X$ induces a Hodge-Lefschetz isomorphism $$L^k: \Ug \otimes_{\Ng} H^{\dim(X)-k}_2(\wt{X},  \mathbb{V}) \to \Ug \otimes_{\Ng} H^{\dim(X)+k}_2(\wt{X},  \mathbb{V}). $$ 
\end{enumerate}

\end{theo}
\begin{proof} Immediate. 
We use the natural quasi-isomorphism $\Ug \otimes _{\Ng} rh_{2}$ to constuct the Hodge complex in the third statement. 
Actually, the cohomology objects are in the abelian category of finitely presented $\Ug$-modules.
 
\end{proof}

\subsection{Using analytic realisations of Hodge Modules}\label{Poin} Assume now that $U\subset X$ is a K\"ahler snc compactification and let $\tilde X \to X$ be a Galois covering space. 
Endow $U$ with a Poincar\'e K\"ahler metric $\omega_U$ \cite{Zuc,CKS,KK}. The lift $\omega_{\wt{U}}$  of $\omega_U$ to
$\wt{U}:=U\times_X \wt{X}$ is a complete K\"ahler metric which is  Poincar\'e with respect to the partial K\"ahler snc compactification
$\wt{U}\subset \wt{X}$. Consider $(U,\mathbb{V}, F^.,S)$ a (say real) polarized VHS on $U$ with quasi unipotent monodromy. 
Recall the  fundamental result  of \cite{CKS,KK} that
the sheafified $L^2$ De Rham complex $\mathcal{DR}^{\bullet}_2(U, \mathbb{V})$  with respect to the Poincar\'e metric on $U$
is a fine model of the perverse sheaf $\mathcal{IC}_X(\mathbb{V})$. 

A trivial modification of the definition  in \cite{KK} replacing $U\to X$ with $\wt{U}\to X$ yields sheaves on $X$  we shall denote by $l^2\pi_* \mathcal{DR}_2^k(\wt{U}, \pi^{-1}\mathbb{V})$
and the operators $D',D'', D$ betwwen these sheaves on $X$ and we have:

\begin{prop}\label{ckskk}
$l^2\pi_* \mathcal{DR}^{\bullet}_2(\wt{U}, \pi^{-1}\mathbb{V})$ is a fine model of $l^2\pi_*\pi^{-1}\mathcal{IC}_X(\mathbb{V})$. 
 \end{prop}
 \begin{prf} Same method as in the proof of Proposition \ref{rhcomp}.
Left to the reader.
 \end{prf}

\begin{coro} Under the current assumptions, the $k$-th cohomology of the $L_2$ De Rham  complex of $\wt{U}$ with values in $\pi^{-1}\mathbb{V}$ 
in the Poincar\'e metric $\omega_{\wt{U}}$ is isomorphic as a $\Ng$-module 
to the $\Ng$-module underlying   $\mathbb{H}^k_2(\wt{X},\pi^{-1}\mathcal{IC}_X(\mathbb{V}))$.
 \end{coro}

\begin{coro}
If the $L^2$ De Rham complex is strongly $\Gamma$-Fredholm, the reduced $k$-th cohomology group twisted by $\Ug$,  $$\Ug\otimes_{\Ng}\mathbb{H}^k_2(\wt{X},\pi^{-1}\mathcal{IC}_X(\mathbb{V}))=
H^k({X}, \Ug\otimes_{\Ng}l^2\pi_*\pi^{-1}\mathcal{IC}_X(\mathbb{V}))$$
carries a natural real $\Ug$-Hodge structure of weight $k+w$. 
 
\end{coro}

\begin{prf} Immediate. 
\end{prf}

We do not have a proof of the $\Gamma$-Fredholmness in this case. 
It is not clear whether the filtrations of this analytically defined Hodge Structure coincide with the algebraic ones we have constructed in this article. 

\begin{rem}
It has been announced in \cite{ShZh} that one can construct for every indecomposable Hodge module $\mathbb{M}$  a  distinguished  complete metric  on the regular part $U$ of its strict support. Distinguished means that  
the sheafified $L^2$ De Rham complex $\mathcal{DR}^{\bullet}_2(U, \mathbb{V})$  with respect to the distinguished metric on $U$
is a fine model of the perverse sheaf $\mathbb{M}^{Betti}$. The above considerations apply to distinguished metrics. 
\end{rem}

\section{Proof of Theorem \ref{thm2}}

Let $X$ be a compact K\"ahler manifold and $\mathfrak{U}$ be  a finite covering of $X$ by sufficiently small Oka-Weil domains. 

\subsection{Direct image by a closed immersion}

Let $i: Z \to X$ be a closed immersion of a smooth compact complex manifold and $(Z, \mathbb{V}, F^., S)$ be a polarized complex Variation of Hodge Structure of weight $w$. The case when $Z=X$ follows
from Theorem \ref{HodgeComplex}.

Then $i^{MHM}_*\mathbb{M}_{Z}(\V)=\mathbb{M}_{i}(\V)$. The  filtered $\Dx$-module $(\mathbb{M}^{DR}_{i}(\V), F)$ can be computed as $i_+ (\mathbb{M}^{DR}_{Z}(\V), F)$. 
$DR(\mathbb{M}^{DR}_{i}(\V))$ is not equal to $i_*DR(\mathbb{M}^{DR}_{Z}(\V))$ except if $Z=X$ where $i_*=Ri_*$ is the ordinary sheaf theoretic direct image. Nevertheless, we can prove:

\begin{lem} $\Ug \otimes_{\Ng}( R\Gamma(X, l^2\pi_* \mathbb{M}^B), (\mathcal{C}^{\bullet}(\mathfrak{U}, l^2\pi_* \mathbb{M}^{DR}_{i}(\V), F), rh_2)$ is a  $\Ug$ Hodge Complex. 
 
\end{lem} 
\begin{proof}

For $q\gg1$, there is a filtered quasi isomorphism $$ F_q(DR(\mathbb{M}^{DR}_{i}(\V)), F)\to( DR(\mathbb{M}^{DR}_{i}(\V)), F)$$ and there is a (differential) filtered quasi-isomorphism of bounded differential complexes  cohenrent sheaves of
$$ i_*DR(\mathbb{M}^{DR}_{Z}(\V)) \to F_q(DR(\mathbb{M}^{DR}_{i}(\V)), F).
$$

On the other hand,  one does have $i_* \V[\dim(Z)]=\mathbb{M}^{B}_{i}(\V)$ and the comparison isomorphism satisfies  $\alpha_{\mathbb{M}_{i}(\V)}= i_*\alpha \circ \eta$.

Consider a locally finite covering $\mathfrak{U}$ of $X$ by small enough Oka-Weil domains. We have a canonical identification of filtered $\Ng$-Fr\'echet complexes
$$(\mathcal{C}^{\bullet}(\mathfrak{U}, l^2\pi_* i_*DR(\mathbb{M}^{DR}_{Z}(\V)), F)= (\mathcal{C}^{\bullet}(i^{-1}\mathfrak{U}, l^2\pi_* DR(\mathbb{M}^{DR}_{Z}(\V))) , F)$$
in particular it is a filtered quasi isomorphism.

It follows that there is a filtered quasi isomorphism 
$$ \Ug \otimes_{\Ng}(\mathcal{C}^{\bullet}(i^{-1}\mathfrak{U}, l^2\pi_* DR(\mathbb{M}^{DR}_{Z}(\V))) , F) \to \Ug \otimes_{\Ng}(\mathcal{C}^{\bullet}(\mathfrak{U}, l^2\pi_* \mathbb{M}^{DR}_{i}(\V), F).
$$
 which is compatible with the comparison isomorphism. Hence, the lemma is a consequence of Theorem \ref{HodgeComplex}. 
\end{proof}

It follows from the construction that the real Hodge structure on $\Ug \otimes_{\Ng} H^{q}(\wt{X}, L^2dR (\mathbb{M}^{DR}_{i}(\V) )$ is the same as the one on $\Ug \otimes_{\Ng} H^{q}(\wt{Z}, L^2dR (\mathbb{M}^{DR}_{Z}(\V) )$. 
The first part of Theorem \ref{thm2} is proved.

\subsection{Comparison with Dingoyan's work} In this subsection, we finish the proof of the second case of Theorem \ref{thm2}. 

\begin{defi} [\cite{DelH3}]
 A Mixed Hodge complex of $\Ug$-modules with real structure  is a triple $((A^{\bullet}, W), (B^{\bullet}, W, F), \gamma)$ 
 where $A^{\bullet}$ is a biregular increasingly filtered complex of $\Ug$-modules with real structure, $\gamma: (A^{\bullet}, W) \to (B^{\bullet},W)$ 
 an isomorphism in the filtered derived category of $\Ug$-modules such that, for all $k\in \Z$,  $$(Gr^k_W A^{\bullet},  (Gr^k_W B^{\bullet}, Gr^k_W F), Gr^k_W\gamma)$$ is 
 Hodge complexes of weight $k$ with real structure. 
\end{defi}

\begin{lem}[\cite{DelH3}] If  $((A, W), (B, W, F), \gamma)$  is a  Mixed Hodge complex of $\Ug$-modules with real structure, for all $n\in \Z$
$$(H^n (A), \mathrm{Im}(H^n(W)\to H^n(A)), H^n(\gamma)^{-1} (F_{H^n (B)}), (H^n(\gamma)^{-1} (F_{H^n (B)}))^{\dagger})$$ 
where $F_{H^n (B)}=\mathrm{Im}(H^n(F) \to H^n (B))$ is a real $\Ug$ mixed Hodge structure. 
 
\end{lem}

\begin{lem}\label{reducpure}
 Let  $X$ be a compact K\"ahler manifold and $\mathbb{M}$ such that the $Gr_{W}^k$ satisfy Conjecture \ref{theconj}. Then $\mathbb{M}$ satisfies Conjecture \ref{theconj}.
\end{lem}

\begin{proof} Under these hypotheses, we see immediately that:
 $$\Ug \otimes_{\Ng}( R\Gamma(X, l^2\pi_* \mathbb{M}^B, W), (\mathcal{C}^{\bullet}(\mathfrak{U}, l^2\pi_* \mathbb{M}^{DR}_{i}(\V), W,  F), rh_2)$$ is a  $\Ug$-Mixed Hodge  Complex. 
\end{proof}

Thanks to \cite{SaiMHM} -see \cite[Example 5.4]{Sch} for one smooth divisor- the second case of Theorem \ref{thm1} follows from the first case and Lemma \ref{reducpure}.  

The third one also follows using some of the properties of  
 of Verdier duality on Mixed Modules, see \cite{SaiMHM}. Indeed $Rj_!j^{-1} \mathbb{M}_X(\mathbb{V})= \mathbb{D} Rj_*j^{-1} {M}_X(\mathbb{V}^{\vee})$ hence $Rj_!j^{-1} \mathbb{M}_X(\mathbb{V})$ is a Mixed Hodge Module. 
Furthermore $\mathbb{D} (Gr_W^k \mathbb{M})= Gr_W^{-k} (\mathbb{D}(\mathbb{M}))$ and $\mathbb{D}(\mathbb{M}_i(\mathbb{V}))= \mathbb{M}_i(\mathbb{V}^{\vee})$.

\section*{Appendix: Fr\'echet Sheaves and the functor $\mathcal{C}$}

\centerline{{\it{\v{C}ech model.}}}

Given $B$ a Banach algebra and $X$ a secound countable locally compact topological space, a $B$-Fr\'echet sheaf is just a sheaf taking its values in
the category of Fr\'echet spaces with a continuous action of $B$. A coherent analytic sheaf on a complex analytic space is a $\C$-Fr\'echet sheaf \cite[Chapter VIII]{GR}, 
in fact a sheaf of Fr\'echet modules over
the structure sheaf which is a 
$\C$-Fr\'echet sheaf of algebras, see \cite{Ho, Sc2}. 

Given $\mathcal{F}$ a $\Gamma$-equivariant
coherent analytic sheaf on a proper $\Gamma$-complex manifold $\wt{X}$, the sheaf
$l^2\pi_*\mathcal{F}$ is a Fr\'echet sheaf of $\Ng$-modules as follows from the construction in \cite{CamDem,Eys} but it is not Montel in the sense of \cite{GR}. 

The sheaf $l^2\pi_*\mathcal{F}$
is not Montel when $\Gamma$ is infinite. There must be a good concept of  $\Gamma$-Montel sheaves (see \cite{Eys} for the corresponding notion of $\Gamma$-compactness and cp. \cite[p. 235]{GR}) but we don't want to try and develop it. 
It will be enough for our present purposes to use the ad hoc theory given in \cite{Eys}.  

For a locally finite covering $\mathfrak{U}$ of $\Gamma \backslash \wt{X}$ by small enough Oka-Weil domains \cite[p. 211]{GR}  we can define
$\mathcal{C}^{\bullet}(\mathfrak{U}, l^2\pi_*\mathcal{F})$ 
the \v{C}ech complex of $l^2\pi_*\mathcal{F}$.  
Here, an open subset $\Omega$ is {\em small enough} if and only if
 the preimage in $\wt{X}$ is a disjoint union of
open subsets finite over $\Omega$.

By a standard application of Leray's theorem,  $\mathcal{C}^{\bullet}(\mathfrak{U}, l^2\pi_*\mathcal{F})$ computes the cohomology of $l^2\pi_*\mathcal{F}$.  

The complex $\mathcal{C}^{\bullet}(\mathfrak{U}, l^2\pi_*\mathcal{F})$ does not depend on $\mathfrak{U}$ in the derived category $D$ of the exact category of
$\Ng$-Fr\'echet modules (cf. Remark \ref{tapia}) hence in the derived category of the abelian category of $\Ng$-modules.  In \cite{Eys} it is proved  that the functor defined $D^b \coh\to DMod(\Ng)$ defined by $\mathcal{F}\mapsto \mathcal{C}^{\bullet}(\mathfrak{U}, l^2\pi_*\mathcal{F})$
at the level of complexes
lifts uniquely and functorially to $D^b(E_f(\Gamma))$ under the natural functor $D^b(E_f(\Gamma)) \to D  Mod(\Ng)$ if $\Gamma \backslash \wt{X}$ is compact. More generally, this holds if $\Gamma\backslash Supp(\mathcal{F})$ is compact.
We denote by $\mathcal{C}: D^b\coh \to D^b(E_f(\Gamma)))$ the resulting functor.

The main ingredient is a construction of a quasi-isomorphism 
$$K^{\bullet}\to \mathcal{C}^{\bullet}(\mathfrak{U}, l^2\pi_*\mathcal{F})$$
from a bounded complex  of projective finite-type Hilbert $\Gamma$-modules $K^{\bullet}$  representing  $\mathcal{C}(\mathcal{F})$. 

To this end,  one uses that the $\Gamma$-Fr\'echet space  $l^2\pi_*\mathcal{F}(\Omega)$ of an open subset $\Omega \subset X$
is the inverse limit of a sequence of Hilbert $\Gamma$-modules with $\Gamma$-compact transition maps  
and the fact that Hilbert $\Gamma$-modules are projective in $E(\Gamma)$. 

Actually one has to do something slightly more complicated, one has a germ at $t=0$ of an increasing family of coverings $\mathfrak{U}_t$ defined for $t\ge 0$ 
such that $\mathfrak{U}=\mathfrak{U}_0$ and: $$C_*(t)=\mathcal{C}^{\bullet}(\mathfrak{U}_t, l^2\pi_*\mathcal{F}) \to \mathcal{C}^{\bullet}(\mathfrak{U}_{t'}, l^2\pi_*\mathcal{F})  $$ is a quasi-isomorphism for $t'\le t$ 
and there there is a germ at $t=0$ 
of a decreasing family of Hilbert $\Gamma$-modules  $(C(t)_{t>0})$  such that $C(t) \to C(t')$ is $\Gamma$-compact for $t>t'$ and $C_*(t)=\varprojlim_{t'<t} C(t')$ if $t>0$. 
The family $(C(t))_{t>0}$ is an inessential auxiliary datum but it is instrumental to the construction of $K^{\bullet}$,  its essential uniqueness and  its functorial properties. 
At least when $\mathcal{F}$ is locally free, 
 $C(t)$ is the subspace of $L_2$ \v{C}ech cochains in $C_*(t)$, $L_2$ being measured with respect to a smooth volume form.

\vskip 0.2cm
\centerline{{\it{Dolbeault model.}}}

In the case where  $\mathcal{F}$ is locally free and $\wt{X}$ is smooth there is a much better way to proceed. Define 
$$Dolb_2^k(\wt{X}, \mathcal {F})=\{ s \in L_{loc}^2(\wt{X}, \mathcal{E}^{0,q} (\mathcal{F})), \ \quad \int_{\wt{X}} \| s\|^2+ \| \bar\partial s \|^2 <+\infty \}$$
where the norms and volume form are computed with respect to a $\Gamma$-equivariant hermitian metric on $\mathcal{F}$ and a $\Gamma$-equivariant hermitian metric on $\wt{X}$

Indeed it is very natural to use the natural map of complexes of $\Ng$-Fr\'echet modules which will be refered to as $CD_2$,  the \v{C}ech-Dolbeault comparison map :

$$ \mathcal{C}^{\bullet}(\mathfrak{U}, l^2\pi_*\mathcal{F}) \to Dolb_2^{\bullet}(\wt{X}, \mathcal {F}))=(\ldots \to Dolb_2^k(\wt{X}, \mathcal {F})) \buildrel{\bar \partial}\over\longrightarrow Dolb_2^{k+1}(\wt{X}, \mathcal {F}) \to \ldots)
$$

attached to a smooth partition of unity $(\phi_{\alpha})$ subordinate to $\mathfrak{U}$
sending the $q$-cochain $(s_{\alpha_{0} \alpha(1) \ldots \alpha(q)})_{|\alpha|=q}$ to the twisted  $(0,q$ form:
$$\sum_{|\alpha|=q} s_{\alpha_{0} \alpha(1) \ldots \alpha(q)} \bar \partial \phi_{\alpha(0)} \wedge \ldots  \bar \partial \phi_{\alpha(q-1)} . \phi_{\alpha(q)}. $$

The complex $Dolb_2^{\bullet}(\wt{X}, \mathcal {F})) $ is complex of separable $\Gamma$-Hilbert modules and the resulting map $K^{\bullet} \to Dolb_2^{\bullet}(\wt{X}, \mathcal {F})) $ is an algebraic isomorphism
induced by a continuous maps at level of the representatives hence $Dolb_2^{\bullet}(\wt{X}, \mathcal {F}))  \in D^b_{E_f(\Gamma)} E_{sep}(\Gamma)$ and is  quasi-isomorphic to $\mathcal{C}(\mathcal{F})\in D^b E_{f}(\Gamma)$. 
\vskip 0.2cm
\centerline{{\it{\v{C}ech-De Rham comparison map}}}

In a similar fashion, if $\mathbb{V}$ is a local system on $X$, and the intersections of elements of $\mathfrak{U}$ are contractible,   we can construct, as in \cite{Dod}, a quasi-isomorphism of separable 
projective Hilbert $\Gamma$-modules in the essential image of $D^bE_f(\Gamma)$ 
 $$CDR_2: \mathcal{C}^{\bullet}(\mathfrak{U}, l^2\pi_*\mathbb{V}) \to L^2DR^{\bullet} (\wt{X}, \mathcal{V}).$$
 
Given a smooth partition of unity $(\phi_{\alpha})$ subordinate to $\mathfrak{U}$, it is defined by
sending the \v{C}ech $q$-cochain $(s_{\alpha_{0} \alpha(1) \ldots \alpha(q)})_{|\alpha|=q}$  to the twisted  $q$ form:
$$\sum_{|\alpha|=q} s_{\alpha_{0} \alpha(1) \ldots \alpha(q)} d \phi_{\alpha(0)} \wedge \ldots  d \phi_{\alpha(q-1)} . \phi_{\alpha(q)}. $$

\vskip 0.2cm

\section*{Concluding remarks}

Assume now $X$ be a complex analytic space that need not be smooth nor reduced. Once one knows that $\dim_{\Ng} H^q_2(\wt{X}, l^2\pi_* \mathcal{F})<+\infty$, e.g. that it lies in the essential image of $E_f(\Gamma)$, 
for all $\mathcal{F} \in \coh$, two purely algebraic properties, 
the special model of the Leray spectral sequence used in \cite[section 6.1]{Eys}, whose main feature is that it comes from a $E_f(\Gamma)$-spectral sequence, 
can be replaced with the usual Leray spectral sequence and L\"uck's theory of $\dim_{\Ng}$ to prove the version of Atiyah's $L_2$-index theorem given in \cite[Theorem 6.2.1]{Eys}. 
The abstract nonsense in \cite{Eys} needed to establish the functoriality of the lift to $E_f(\Gamma)$ can also be eliminated. 

We conclude by  stating a weaker form Theorem \ref{thm1} in terms of finite dimensional $\Ng$-modules, thus losing the information contained in Novikov-Shubin invariants. It follows from
the part of the proof of Theorem \ref{thm1} where global good filtrations are used (subsections \ref{vncohd}, \ref{liftrh}):

\begin{theoi}

Let $X$ be a complex manifold and $\widetilde{X} \to X$ be a Galois covering with Galois group $\Gamma$. 
Let $MD'(X)$  be the abelian category  whose objects are triples $$\mathbb{M}=(\mathcal{M}=\mathbb{M}^{DR}, P=\mathbb{M}^{Betti}, \alpha)$$ where $\mathcal{M}$ is a holonomic $\Dx$-module
supported on a $\Gamma$-cocompact analytic subspace,
 $P$ is a perverse sheaf of $\R$-vector spaces and $\alpha: P\otimes_{\R}\C \to DR(\mathcal{M})$ is an isomorphism in the derived category of sheaves and whose morphisms are the obvious ones.

There is a  
 $\partial$-functor which, on the Betti side, is compatible with proper direct images, satisfies Atiyah's $L_2$ index theorem:  
$$ L_2dR: D^b MD'(X) \to \mathrm{EssentialImage} (D^b(E_f(\Gamma)) \buildrel{\Phi}\over\to D^b Mod(\Ng))
$$
and for each $\mathbb{M}\in MD(X)$ and $q\in \Z$  functorial isomorphisms in $E_f(\Gamma)$ $$H^q(L_2dR(\mathbb{M})) \cong\mathbb{H}^q_{(2)}(\widetilde{X}, \mathbb{M}^{Betti})\cong\mathbb{H}^q_{DR, (2)}(\widetilde{X}, \mathbb{M}^{DR}).$$ 

The functor $\Ug\otimes_{\Ng}L_2dR : D^b MD'(X) \to  D^b Mod(\Ug)$ takes its values in the essential image of the the category of  complexes of finite type projective $\Ug$-modules, 
is compatible with proper direct images, satisfies Atiyah's $L_2$ index theorem and Poincar\'e-Verdier duality. 
\end{theoi}

\end{document}